\documentclass[a4paper]{amsart}
\usepackage[utf8]{inputenc}

\usepackage{amssymb}
\usepackage{amsthm,mathtools}
\usepackage[active]{srcltx}
\usepackage{color}
\usepackage{enumerate}
\usepackage[all]{xy}
\usepackage{enumitem}
\usepackage{float}
\usepackage{float}
\usepackage{adjustbox}

\renewcommand{\k}{\Bbbk}
\newcommand{\ku}{\Bbbk}

\newcommand{\nc}{\newcommand}

\newcommand{\ot}{\otimes}

\newcommand{\ydh}{{}^{H}_{H}\mathcal{YD}}

\newcommand{\rg}{\rangle}
\renewcommand{\lg}{\langle}
\def\qb{\mathfrak{q}}
\newcommand{\bq}{\mathfrak{q}}
\newcommand{\ra}{\rightharpoonup}
\newcommand{\la}{\leftharpoonup}

\newcommand{\Hom}{\operatorname{Hom}}

\newcommand{\ex}{\operatorname{Exp}}
\newcommand{\id}{\operatorname{id}}
\newcommand{\img}{\operatorname{img}}

\newcommand{\can}{\operatorname{can}}

\newcommand{\Ss}{\mathcal{S}}

\newcommand{\Alg}{\operatorname{Alg}}

\newcommand{\Cleft}{\operatorname{Cleft}}

\newcommand{\co}{\operatorname{co}}
\renewcommand{\mod}{\operatorname{-mod}}
\newcommand{\bimod}{\operatorname{-bimod}}

\newcommand{\N}{\mathbb{N}}

\newcommand{\s}{\mathbb{S}}

\newcommand{\I}{\mathbb{I}}

\newcommand{\mE}{\mathcal{E}}
\renewcommand{\P}{\mathrm{P}}

\newcommand{\A}{\mathcal{A}}
\newcommand{\Bq}{\mathfrak{B}}

\newcommand{\eps}{\epsilon}

\def\pf{\begin{proof}}
\def\epf{\end{proof}}

\def\bs{\boldsymbol}

\renewcommand{\H}{\operatorname{H}}
\newcommand{\Z}{\operatorname{Z}}
\newcommand{\B}{\operatorname{B}}
\newcommand{\C}{\operatorname{C}}

\newcommand{\Dchaintwo}[3]{\xymatrix@C-4pt{\overset{#1}{\underset{\ }{\circ}}\ar
		@{-}[r]^{#2}
		& \overset{#3}{\underset{\ }{\circ}}}}

\nc{\ben}{\begin{enumerate}[(i)]}
\nc{\een}{\end{enumerate}}

\numberwithin{equation}{section}

\theoremstyle{plain}
\newtheorem{theorem}{Theorem}[section]
\newtheorem{lemma}[theorem]{Lemma}

\newtheorem{proposition}[theorem]{Proposition}
\newtheorem{corollary}[theorem]{Corollary}
\newtheorem{definition}[theorem]{Definition}
\newtheorem*{theoremintro}{Theorem}

\theoremstyle{remark}
\newtheorem{remark}[theorem]{Remark}

\newtheorem*{acknowledgement*}{Acknowledgement}

\newtheorem{example}[theorem]{Example}

\allowdisplaybreaks

\title[Explicit Hopf cocycles]{On the computation of Hopf 2-cocycles, with an example of diagonal type}

\author{Agust\'in Garc\'ia Iglesias}
\address{A.G.I.: FaMAF-CIEM (CONICET),
	Medina Allende S/N,
	Universidad Nacional de C\'ordoba,
	Ciudad Universitaria, C\'ordoba (X5000HUA),
	República Argentina. }
\email{agustingarcia@unc.edu.ar}

\author{José Ignacio Sánchez}
\address{J.I.S.: FaMAF-CIEM (CONICET),
	Medina Allende S/N,
	Universidad Nacional de C\'ordoba,
	Ciudad Universitaria, C\'ordoba (X5000HUA),
	República Argentina. }
\email{jose.ignacio.sanchez@mi.unc.edu.ar}

\begin{document}

\begin{abstract}
We present a framework for the computation of the Hopf 2-cocycles involved in the deformations of Nichols algebras over semisimple Hopf algebras. We write down a recurrence formula and investigate the extent of the connection with invariant Hochschild cohomology in terms of exponentials. 

As an example, we present detailed computations leading to the explicit description of the Hopf 2-cocycles involved in the deformations of a Nichols algebra of Cartan type $A_2$ with $q=-1$, a.k.a.~the positive part of the small quantum group $\mathfrak{u}^+_{\sqrt{\text{-1}}}(\mathfrak{sl}_3)$. We show that these cocycles are generically {\it pure}, that is they are not cohomologous to exponentials of Hochschild 2-cocycles.
\end{abstract}

\maketitle

\section{Introduction}\label{sec:intro}

Let $A=(A,m,\Delta)$ be a Hopf algebra over a field $\k$. A Hopf cocycle $\sigma$ with coefficients in $\k$ for $A$ is a linear map $\sigma\colon A\ot A\to \k$ which is convolution invertible and satisfies
\[
\sigma(x_{(1)}, y_{(1)}) \sigma(x_{(2)} y_{(2)}, z) =
\sigma(y_{(1)}, z_{(1)}) \sigma(x, y_{(2)}z_{(2)}), \quad x,y,z\in A.
\]

This condition allows to deform the multiplication $m\rightsquigarrow m_\sigma=\sigma\ast m\ast \sigma^{-1}$ into a new associative product, so that the {\it cocycle deformation} $A_\sigma=(A,m_\sigma,\Delta)$ is again a Hopf algebra.
Introduced by Doi and Takeuchi in \cite{DT}, Hopf cocycles have evolved  in recent years from a technical device to an inescapable and key tool in the classification program for non-semisimple Hopf algebras.

Indeed, the classification of finite-dimensional pointed Hopf algebras over abe\-lian groups is completely determined by Hopf cocycles. Each such Hopf algebra is a cocycle deformation of the bosonization $B\# H$, where  $B=\Bq(V)$ is a finite-dimensional Nichols algebra of diagonal type and $H=\k\Gamma$ is a commutative group algebra. This was first observed by Masuoka in \cite{Ma}, see also \cite{GrM}, for the Hopf algebras classified in \cite{AS}, and proved in full generality in \cite{AnG}. A similar statement is true for every known example of a finite-dimensional pointed Hopf algebra over $H=\k G$, $G$ non-necessarily abelian, cf.~\cite{GaM,GIM,GIV1,GIV2}. Moreover, the same holds for copointed Hopf algebras over $H=\k^G$, see \cite{AV,FGM,GIV2}.

It is important to remark that, despite the name, Hopf cocycles do not arise from a cohomology setup (unless the Hopf algebra is cocommutative, in which case there are no deformations in this sense).
In fact, the set $Z^2(A)$ of Hopf cocycles in $A$
is not a group. This is a major obstruction for a systematic construction of cocycles.
There is, however, a notion of cohomologous cocycles, via an action of the group of convolution units in $A^\ast$.

As a matter of fact, the classification results bypass an explicit description of the cocycles involved and work with the algebraic counterpart of cleft objects.
These are certain comodule algebras $C$ endowed with a convolution invertible comodule isomorphism $\gamma\colon A \to C$. Any pair $(C,\gamma)$ defines a Hopf cocycle via
\begin{align}\label{eqn:sigmagamma-intro}
\sigma(a,b)=\gamma(a_{(1)})\gamma(b_{(1)})\gamma^{-1}(a_{(2)}b_{(2)}), \qquad a,b\in A.
\end{align}

There is a more cohomological alternative to compute Hopf cocycles, rooted in the Hochschild cohomology $\H^{\bullet}(A,\k)$ of $A$, as some Hopf cocycles can be obtained as exponentials of Hochschild 2-cocycles in $\Z^2(A,\k)$.  We also investigate this approach, which turns out to be rather limited.
We say that a Hopf cocycle is {\it pure}  when it is not cohomologous to an exponential.
%
%

We propose to explicitly compute the Hopf cocycles involved in the classification by writing down the section for each cleft object and calculating the values of $\sigma$ using \eqref{eqn:sigmagamma-intro}. Also, we determine which of these cocycles are pure and which are exponentials.

\subsection{Setting}\label{sec:setting}
Motivated by the state-of-the-art of the classification program just described, we restrict ourselves to the following context
\begin{itemize}[leftmargin=*]
	\item $A=B\# H$, where $H$ is a semisimple  Hopf algebra and $B\in\ydh$ is a (braided) graded Hopf algebra $B=\bigoplus_{n\geq 0} B_n$ with $B_0=\k$ and generated by $V\coloneqq B_1$. That is, $B$ is a pre-Nichols algebra.
	\item The Hopf cocycles $\sigma$ for $A$ belong to the set
	\[
	Z^2_0(A)\coloneqq\{\sigma\in Z^2(A) | \sigma_{|H\otimes H}=\eps\}.
	\]
	In particular, they are determined by the set
	\[
	Z^2(B)^H\coloneqq\{\sigma_B\in Z^2(B) | \sigma(h_{(1)}\cdot x,h_{(2)}\cdot y)=\eps(h)\sigma(x,y), \  h\in H, x,y\in B\}
	\]
	of $H$-linear cocycles $\sigma_B$ for $B$, as a braided bialgebra, see \eqref{eqn:hopfcocycle-cordfree}, so that
	\begin{align}\label{eqn:sigma-intro}
	\sigma(ah,bk)=\sigma_B(a,h\cdot b)\eps(k), \qquad a,b\in B, \ h,k\in H.
	\end{align}
	Thus, $\sigma_B=\sigma_{|B\otimes B}$ and $\sigma$ is determined by its values on a basis of $B$. We write $\sigma=\sigma_B\# \eps$, so $Z^2_0(A)\simeq Z^2(B)^H$. See \cite{AG} for details on braided cocycles.
\end{itemize}

We develop some tools that reduce the number of calculations we need to perform to compute the cocycles, see Lemma \ref{lem:descomposi}.

\c{S}tefan shows in \cite{St} that there is an action of $H$ on $\Z^\bullet(B,M)$, that descends to $\H^{\bullet}(B,M)$, $M\in B\bimod$, in such a way that $\H^\bullet(A,M)\simeq \H^\bullet(B,M)^H$. This action is not explicit in {\it loc.cit.} We write down \c{S}tefan's Theorem in our context as Theorem \ref{thm:stefan}: we present an element-based proof of it and show that the action is:
\begin{align*}
(h\ra f)({\bf x})=h_{(1)}f(\Ss(h_{(2)})\cdot {\bf x})\Ss(h_{(3)}),
\end{align*}
for $h\in H$, $f\in \Z^q(B,M)\subset\Hom_\k(B^{\ot\,q}, M)$ and ${\bf x}\in B^{\ot\,q}$; $\Ss$ denotes the antipode of $H$. See also \cite{CGW} for $H$ a group algebra.
As well, we show that we can restrict the study of $\eta\in\Z^2(A,\k)$ to cocycles in $\eta_B\in \Z^2(B,\k)^H$, {\it cf.}~ \eqref{eqn:sigma-intro}. Finally, we show in Proposition \ref{pro:invariance} that $e^\eta\in Z^2(B)^H$ if and only if $\eta\in \Z^2(B,\k)^H$.

Now, the cleft objects in the classification arise as $C=E\# H$, for $E$ a cleft object for $B$ in $H\mod$. In particular the section $\gamma$ is given by $\gamma=\gamma_B\#\id_H$, $\gamma_B\colon B\to E$. Formula \eqref{eqn:sigmagamma-intro} extends to this context, see \cite{AG}.

In this setting, our strategy runs as follows.
\begin{itemize}
	\item[(a)] Compute the section $\gamma_B\colon B\to E$ in a basis $\mathbb{B}$ of $B$.
	\item[(b)] Write down the table of values of $\sigma$ as in \eqref{eqn:sigmagamma-intro} in terms of $\mathbb{B}^2$.
	\item[(c)] Compute the space $\Z^2(B,\k)^H$ of $H$-invariant Hochschild 2-cocycles and find necessary and sufficient conditions so that $e^\eta$ is a Hopf cocycle.
	\item[(d)] Determine those Hopf cocycles that are cohomologous to exponentials $e^\eta$ and those that are pure.
\end{itemize}

\subsection{An example of diagonal type}
We perform all of these steps for a braiding \[
\qb=\left(\begin{smallmatrix}
-1&q_{12}\\q_{21}&-1
\end{smallmatrix}\right)\] of Cartan type $A_2$, with parameter $q=q_{12}q_{21}=-1$, and a principal YD-realization, see \eqref{eqn:realization}, in $\ydh$ for a semisimple Hopf algebra $H$.
We recall that
\[\mathbb{B}=\{1,x_1,x_2,x_{12},x_2x_1, x_2x_{12}, x_{12}x_1, x_2x_{12}x_1\}
\] is a PBW basis for the Nichols algebra $B=\Bq_{\qb}$ and that the deformations $\mathfrak{u}_\qb(\bs\lambda)$ of $A=B\# H$ are classified by triplets $\bs\lambda=(\lambda_1,\lambda_2,\lambda_{12})\in\k^3$, subject to certain restrictions (depending on $H$), see \S \ref{sec:a2}.
Our main results are Theorems \ref{thm:sigma-a2} and \ref{thm:sigma-pure}, which can be summarized as follows.
\begin{theoremintro}
Let $\qb$, $H$ be as above.
Any Hopf cocycle $\sigma\in Z^2(\Bq_{\qb})^H$ is cohomologous to a cocycle
$\sigma_{\bs\lambda}$, $\bs\lambda=(\lambda_1,\lambda_2,\lambda_{12})$,
 as in the following table:
	\begin{table}[H]
		\resizebox{12cm}{!}
		{\begin{tabular}{|c|c|c|c|c|c|c|c|c|}
				\hline
				$\sigma_{\bs\lambda}$ &  $x_1$ & $x_2$ & $x_{12}$ & $x_2x_1$  &  $x_2x_{12}$ & $x_{12}x_1$ & $x_2x_{12}x_1$ \\
				\hline
				$x_1$   & $\lambda_1$ & $0$ & $0$ & $0$ & $\lambda_{12}$ & $0$ & $0$ \\
				\hline
				$x_2$  & $0$ & $\lambda_2$ & $0$ & $0$  &  $0$ & $2q_{12}\lambda_1\lambda_2$& $0$\\
				\hline
				$x_{12}$  & $0$ & $0$ & $\lambda_{12}$ & $0$ &  $0$ & $0$ & $0$ \\
				\hline
				$x_2x_1$  & $0$ & $0$ & $0$ & $-q_{21}\lambda_1\lambda_2$  &  $0$ & $0$ & $0$   \\
				\hline
				$x_2x_{12}$  & $0$ & $0$ & $0$ & $0$& $-q_{12}\lambda_2\lambda_{12}$ & $0$ & $0$  \\
				\hline
				$x_{12}x_1$  & $0$ & $2q_{12}\lambda_1\lambda_2 + \lambda_{12}$ & $0$ & $0$ &  $0$ & $q_{12}\lambda_{12}\lambda_1+4\lambda_1^2\lambda_2$ & $0$  \\
				\hline
				$x_2x_{12}x_1$ & $0$ & $0$ & $0$ & $0$  &  $0$ & $0$ & $q_{12}\lambda_2\lambda_{12}\lambda_1$ \\
				\hline
		\end{tabular}}
	\end{table}
In particular, 
\[
\mathfrak{u}_\qb(\bs\lambda)\simeq (\Bq_{\qb}\# H)_{\sigma_{\bs\lambda}\# \eps}.
\]

In turn, the cocycle $\sigma_{\bs\lambda}$ is cohomologous to an exponential if and only if at most one of the parameters $\lambda_1,\lambda_2,\lambda_{12}$ is nonzero.\qed
\end{theoremintro}
This table was computed by the second author in his {\it Tesis de Licenciatura} \cite{S}.

\subsection{About the exponential}
The theorem above shows in particular that the exponential approach is far from optimal. Up to cohomology, only three non-trivial cocycles are exponentials, namely $\sigma_{(1,0,0)}, \sigma_{(0,1,0)}$ and $\sigma_{(0,0,1)}$.

On the other hand, it is unclear which conditions on $\eta$ are necessary to imply that the exponential $e^\eta$ is a Hopf cocycle. There are  well-known sufficient commutation conditions \eqref{eqn:conm1} and \eqref{eqn:conm2} in the convolution algebra $\Hom(A^{\ot 3},\k)$. As a byproduct of our calculations, we obtain in Lemma \ref{lem:converse} that
\begin{itemize}[leftmargin=*]
	\item There are Hochschild cocycles $\eta$ that fail \eqref{eqn:conm1}, \eqref{eqn:conm2} but still $e^\eta$ is a Hopf cocycle. 
	\item However, we show that in our examples such cocycles $\eta$ are cohomologous to a cocycle satisfying those conditions.
	We are not aware if this a general fact.
\end{itemize}

An important comment worth making is that,
depending on the realization and certain normalization conditions, a pure cocycle deformation $A_\sigma$ can be isomorphic to an exponential deformation $A_{e^\eta}$, see Remark \ref{rem:lambda=0}. Nevertheless, the corresponding cleft objects are not isomorphic.

Some of these observations have already been implicitly pointed in \cite{GaM}, for cocycles related to pointed Hopf algebras over $\s_4$.
Nevertheless, the exponential Hopf cocycles obtained in {\it loc.cit.}~considered a restricted class of Hochschild cocycles, namely  those that are concentrated in degree one. We hope to obtain a full description of the missing pure cocycles in future work.

On the other hand, those restrictions are enough to find all cocycles for pointed Hopf algebras over $\s_3$ {\it ibid.}, as deformations of the Fomin-Kirillov algebra $\mathcal{FK}_3$ and copointed Hopf algebras over the dihedral groups $\mathbb{D}_{4m}$ in \cite{FGM}.
In a forthcoming work \cite{GS}, we show that pure cocycles are indeed needed for pointed deformations of $\mathcal{FK}_3$ over more general groups  $G\twoheadrightarrow\s_3$, or for copointed deformations over $\k^{\s_3}$. 

It is also worth mentioning that, under certain conditions, some families of cocycles for the pointed Hopf algebras in \cite{AS} were described in \cite{MW}. The example we treat in this article escapes those restrictions.

\subsection{Organization}

The paper is organized as follows. In \S \ref{sec:prels}, we recall some preliminaries on Hopf algebras and deformations. In \S \ref{sec:tools} we develop some tools and techniques to compute the cocycles. In \S \ref{sec:hoch} we study the Hochschild cohomology of the algebras in our setting \S \ref{sec:setting}, in particular the structure of the set of $H$-invariant 2-cocycles. We also investigate the $H$-invariance introduced in \cite{St} and give an explicit description of the action.
We combine these analysis to study the exponential cocycles in \S \ref{sec:hopfschild}. Finally, we apply all of these steps in \S  \ref{sec:a2} where we compute in detail the cocycles associated to a deformation of Cartan type $A_2$.

\section*{Acknowledgments}
We thank Andrea Solotar for conversations in early stages of this work, particularly about the contents of \cite{St}.

\section{Preliminaries}\label{sec:prels}

We work over an algebraically closed field of characteristic zero $\k$. All tensor products, vector spaces and maps are assumed to be defined over $\k$. If $\theta\in\N$, we set $I_\theta=\{1,\dots,\theta\}\subset \N$.

We let $(H,m,\Delta)$ be a Hopf algebra; we use Sweedler's notation $\Delta(h)=h_{(1)}\ot h_{(2)}$ for the comultiplication, similarly for comodules over $H$ and the  coaction. We set $\Delta^{(2)}=(\Delta\ot\id)\Delta(=(\id\ot\Delta)\Delta)$.
We write $G(H)$ for the set of group-like elements on $H$, $\P(H)$ for the primitive elements and $U(H)$ for the units. We denote by $H_0$ the coradical of $H$. Recall that $H$ is pointed when $H_0=\k G(H)$ and copointed when $H_0=k^G$ for some (finite) non-abelian group $G$.

A (right) cleft object for $H$ is a right comodule algebra $C$ endowed with a convolution invertible comodule isomorphism $\gamma\colon H\to C$. This map is called a section when $\gamma(1)=1$. We let $\Cleft(H)$ denote the set of cleft objects up to isomorphism. Left, and bi-cleft, objects are defined similarly.

\subsection{Hopf cocycles}

A Hopf 2-cocycle $\sigma: H\ot  H \to \ku$ is a
convolution-invertible linear map $x\ot y\mapsto\sigma(x,y)$ satisfying
\begin{align}\label{eqn:hopfcocycle}
\sigma(x_{(1)}, y_{(1)}) \sigma(x_{(2)} y_{(2)}, z) &=
\sigma(y_{(1)}, z_{(1)}) \sigma(x, y_{(2)}z_{(2)}), \quad x,y,z\in H.
\end{align}
It follows that $\sigma(1,1)\neq 0$ and $\sigma(x, 1) = \sigma(1, x) = \eps(x)\sigma(1,1)$, $x\in H$, so by
\begin{align}\label{eqn:normalized}
\sigma\leftrightarrow \sigma(1,1)^{-1}\sigma
\end{align}
we can always assume that $\sigma(x, 1) = \sigma(1, x) = \eps(x)$. Such cocycles are called {\it normalized}.
We shall denote by
\[Z^2(H)=\{\sigma\colon H\ot H\to \k \ | \ \text{\eqref{eqn:hopfcocycle} holds}\}\] the set of normalized Hopf 2-cocycles in $H$.

We will use a coordinate-free version of \eqref{eqn:hopfcocycle}, namely
\begin{align}\label{eqn:hopfcocycle-cordfree}
(\sigma\ot\eps)\ast \sigma(m\ot\id)&=
(\eps\ot\sigma)\ast \sigma(\id\ot\, m).
\end{align}
This allows us to consider a braided version of Hopf cocycles for any bialgebra in a braided category; see \cite{AG} for details.

%
%

\subsubsection{Cocycle deformations}

Let $\sigma\in Z^2(H)$. Then $\cdot_{\sigma}:H\ot H\rightarrow H$, given by
\begin{align}\label{eqn:cociclo-prod}
x\cdot_{\sigma}y &= \sigma(x_{(1)}, y_{(1)}) x_{(2)} y_{(2)}
\sigma^{-1}(x_{(3)}, y_{(3)}), \quad x, y \in H,
\end{align}
defines an associative product on the vector space $H$. Moreover,
the collection $(H,\cdot_\sigma,  \Delta)$ is a Hopf
algebra with antipode $\Ss_{\sigma} = f*\Ss*f^{-1}$, for $f =
\sigma\circ(\id\ot\,\Ss)\circ\Delta$. This Hopf algebra is denoted $H_\sigma$ and referred to as a {\it cocycle deformation of $H$}.

If $\sigma\in Z^2(H)$, then there is another deformation of the multiplication in $H$. Namely, $\cdot_{(\sigma)}:H\ot H\rightarrow H$, given by
\begin{align}\label{eqn:cociclo-prod-cleft}
x\cdot_{(\sigma)}y &= \sigma(x_{(1)}, y_{(1)}) x_{(2)} y_{(2)}, \quad x, y, \in H,
\end{align}
also defines an associative product on the vector space $H$; we denote this new algebra  by $H_{(\sigma)}$.
In this case,
the comultiplication $H\to H\otimes H$ defines coactions $H_{(\sigma)}\to H_{(\sigma)}\otimes H$
and $H_{(\sigma)}\to H_{\sigma}\otimes H_{(\sigma)}$ in such a way that $C=H_{(\sigma)}$ becomes a $(H_{\sigma},H)$-bicleft object.

Moreover, any (right) cleft object $C$ for $H$ (and hence any cocycle deformation $H_{\sigma}$) arises in this way. If $C$ is a cleft object with a normalized section $\gamma:H\to C$, then $C\simeq H_{(\sigma)}$; for the cocycle $\sigma\colon H\ot H\to \k$ determined by the formula
\begin{align}\label{eqn:formula-cleft}
\sigma(x,y)=\gamma(x_{(1)})\gamma(y_{(1)})\gamma^{-1}(x_{(2)}y_{(2)}), \quad x,y\in H.
\end{align}
We stress that formula \eqref{eqn:formula-cleft} extends to cleft objects in braided categories, see \cite{AG}.

Furthermore, it is possible to recover $H_{\sigma}$ completely in terms of the cleft object $C$, as the so-called Schauenburg left Hopf algebra $L(C,H)$; in particular, no cocycles are explicitly involved. See \cite{Sch} for further details.

\subsubsection{Cohomologous Hopf cocycles}

The group $U(H^*)$ of convolution units in $H^\ast$ acts on $Z^2(H)$ via
\begin{equation}\label{eqn:hopfcohomologous}
(\alpha\rightharpoonup\sigma)(x,y)=\alpha(x_{(1)})\alpha(y_{(1)})\sigma(x_{(2)},y_{(2)})\alpha^{-1}(x_{(3)}y_{(3)}).
\end{equation}
We write $\sigma\sim_\alpha\sigma'$ when $\sigma'=\alpha\rightharpoonup\sigma$
and set
\[
H^2(H)\coloneqq Z^2(H)/U(H^\ast)=\{[\sigma]:\sigma\in Z^2(H)\}.
\]

Moreover, if $\sigma'=\alpha\rightharpoonup\sigma$, then $H_\sigma\simeq H_{\sigma'}$ as Hopf algebras. However, the converse is not true, see Remark \ref{rem:taft} below. 

The equivalence does hold for $H$-cleft objects, namely $(C,\gamma)\simeq (C',\gamma')$ if and only if $[\sigma]=[\sigma']$, for $\sigma,\sigma'$ the Hopf 2-cocycles associated to $\gamma,\gamma'$ respectively as in \eqref{eqn:formula-cleft}, so that $C\simeq H_{(\sigma)}$ and $C'\simeq H_{(\sigma')}$, see \cite[Theorem 2.2]{D}. That is, there is an equivalence
\[
\Cleft(H)\longleftrightarrow H^2(H).
\]

We may always assume that such $\alpha$ as in \eqref{eqn:hopfcohomologous} is such that $\alpha(1)=1$.
Indeed, $c=\alpha(1)\neq 0$ as $\alpha$ is convolution invertible and
\[
c^{-1}\alpha\rightharpoonup \sigma=c^{-1}(\alpha\rightharpoonup \sigma)\sim_{c^{-1}\eps} \alpha\rightharpoonup \sigma.
\]

\begin{remark}\label{rem:taft}
	Two cocycle deformations $H_{\sigma}$ and $H_{\sigma'}$ may be isomorphic for two non-equivalent cocycles $\sigma,\sigma'\in Z^2(H)$. 
	
	We briefly develop an example. Let $H=\k\lg x,g:x^2=0, g^2=1, gx=-xg\rg$ be the Taft algebra and let $\sigma_\lambda$, $\lambda\in\k$, be the cocycle determined by $\sigma_\lambda(x,x)=\lambda$ and 
	\[
	\sigma_\lambda(x^a g^i, x^bg^j)=(-1)^{bi}\sigma_\lambda(x^a,x^b),
	\]
	for $a,b,i,j\in\{0,1\}$. It follows that $H_{\sigma_{\lambda}}\simeq H$ for every $\lambda\in\k$. If $\lambda\neq 0$, then $\sigma_{\lambda}\not\sim \sigma_1$, but $\sigma_{1}\not\sim \sigma_0=\eps$. Indeed, they determine the (non-isomorphic) cleft objects $C_\lambda=\k\lg y,g:y^2=\lambda, g^2=1, gy=-yg\rg$, $\lambda=0,1$.
\end{remark}

\subsection{Yetter-Drinfeld modules and Nichols algebras}

We recall that the category $\ydh$ of Yetter-Drinfeld modules over $H$ is the category of all simultaneously  left modules and left comodules $M$ such that the action $\cdot$ and the coaction $\rho$ are subject to
\[
\rho(h\cdot m)=h_{(1)}m_{(-1)}\Ss(h_{(3)})\ot h_{(2)}\cdot m_{(0)}, \qquad m\in M, h\in H.
\] When the antipode $\Ss$ of $H$ is bijective this is a braided category, with braiding $c_{M,N}\colon M\ot N\to N\ot M$, $m\ot n\mapsto m_{(-1)}\cdot n\ot m_{(0)}$; $M,N\in\ydh$, $m\in M$, $n\in N$.

A braided vector space is a pair $(V,c)$, where $V$ is a vector space together with an invertible linear map $c\colon V\ot V\to V\ot V$ that satisfies the braid equation $(c\ot\id)(\id\ot\, c)(c\ot\id)=(\id\ot\, c)(c\ot\id)(\id\ot \,c)$. In particular, any $M\in\ydh$ is a braided vector space with $c=c_{M,M}$. Conversely, a realization of $(V,c)$ over $H$ is a structure of Yetter-Drinfeld module on $V$ so that $c=c_{V,V}\in\ydh$.

If $B\in\ydh$ is an algebra, then one may use the braiding to define an algebra structure on $B\ot B\in\ydh$.
A bialgebra, resp.~Hopf algebra, in $\ydh$ is defined under this convention. In this case, the bosonization $B\# H$ becomes a (Hopf) algebra. We recall that in this case the multiplication is defined via
\[
(b\otimes h)(c\otimes k)=b (h_{(1)}\cdot c)\ot h_{(2)}k, \qquad b,c\in B, \ h,k\in H.
\]
We write $bh$ for an element $b\#h$ of $B\# H$. Notice that this is coherent with the multiplication just defined as
\[B\simeq B\otimes 1_H\subset B\# H \qquad  \text{ and }  \qquad H\simeq 1_B\ot H\subset B\otimes H
\] as algebras and $(b\otimes 1)(1\otimes h)=b\otimes h$.
As well, notice that, for $h\in H$, $b \in B$:
\begin{align}\label{eqn:hb}
hb&=h_{(1)}\cdot b h_{(2)}, & bh&=h_{(2)}(\Ss^{-1}(h_{(1)})\cdot b), & h_{(1)}b\Ss(h_{(2)})&=h\cdot b.
\end{align}
In turn, the comultiplication is given by
the rule
\begin{align}\label{eqn:comultBH}
\Delta(x\#h)=x_{(1)}\#x_{(2)(-1)}h_{(1)}\ot x_{(2)(0)}\# h_{(2)}.
\end{align}
See \cite{AS1} for details.

A Nichols algebra is a connected graded Hopf algebra $B=\oplus_{n\geq 0} B_n$ in $\ydh$ generated by $V=B_1$ and such that $\P(B)=V$.
We say that such a Hopf algebra is pre-Nichols when this last condition is not satisfied, so that $V\subsetneq \P(V)$.
We denote such Nichols algebra by $\Bq(V)$, as it is completely determined by the braided vector space $(V,c)$, independently of the realization in $\ydh$.

A finite-dimensional braided vector space is of diagonal type when there is a basis $\{x_1,\dots,x_\theta\}$ of $V$ and a matrix $\qb=(q_{ij})_{i,j\in\I_\theta}$ so that the braiding $c=c^{\qb}$ is given by $c(x_i\ot x_j)=q_{ij}x_j\ot x_i$, $i,j\in \I_\theta$. In this setting $\Bq(V)$ is also denoted $\Bq_{\qb}$.

A principal YD-realization of $V$ over $H$ is a family $(\chi_i,g_i)_{i\in \I_\theta}$ with $g_i\in G(H)$ and $\chi_i\colon H\to \k$  algebra maps so that 
\begin{align*}
\chi_j(g_i)=q_{ij}  \quad \text{ and } \quad  \chi_i(h)g=\chi_i(h_{(2)})h_{(1)}g\Ss(h_{(3)}), 
\end{align*}
$i,j\in \I_\theta$, $g\in G(H), h\in H$. This determines a realization $V\in\ydh$ via
\begin{align}\label{eqn:realization}
\rho(x_i)&=g_i\ot x_i, &  h\cdot x_i&=\chi_i(h)x_i, \qquad h\in H, i\in\I_\theta.
\end{align}

\begin{remark}
Diagonal braided vector spaces $(V,c^{\qb})$, or equivalently matrix braidings $\bq$, with finite-dimensional associated Nichols algebra, are classified in \cite{H} in terms of decorated Dynkin diagrams. On top of that, the Nichols algebras $\Bq_{\qb}$ are presented by generators and relations in \cite{An1,An2} as quotients of the tensor algebra $T(V)$. In turn, if $H$ is a semisimple Hopf algebra with a principal YD-realization $V\in\ydh$, then all deformations of $A=\Bq_{\qb}\# H$ are classified in \cite{AnG}, by defining an exhaustive family of cleft objects $(C_{\bs\lambda})_{\bs\lambda\in\Lambda}$ for $A$ in such a way that every deformation arises as the associated Schauenburg left Hopf algebra $L(C_{\bs\lambda},A)$, which is presented by generators and relations as a quotient $\mathfrak{u}_{\qb}(\bs\lambda)$ of $T(V)\# H$. In particular, $\mathfrak{u}_{\qb}(\bs\lambda)$ is a cocycle deformation of $A$.
\end{remark}

For any connected $B=\bigoplus_{n\geq 0} B_n$ bialgebra and $b\in B^+\coloneqq\bigoplus_{n\geq 1} B_n$, we shall consider the restricted comultiplication:
\begin{align}\label{eqn:restr-comult}
\underline{\Delta}(b)\coloneqq\Delta(b)-b\ot 1-1\ot b.
\end{align}
Also, we set $\underline{\Delta}(1)=1\ot 1$; and write
$
\underline{\Delta}(b)=b_{\underline{{(1)}}}\ot b_{\underline{{(2)}}}$, $b\in B.
$
We will use a similar notation $\underline{\rho}$ for coactions. A restricted version of coassociativity holds:
	\[
	(\underline{\Delta}\ot \id)\underline{\Delta}=(\id\ot \underline{\Delta})\underline{\Delta}.
	\]
	As usual, we write $b_{\underline{{(1)}}{\underline{{(1)}}}}\ot b_{\underline{{(1)}}{\underline{{(2)}}}}\ot b_{\underline{{(2)}}}=b_{\underline{{(1)}}}\ot b_{\underline{{(2)}}{\underline{{(1)}}}}\ot b_{\underline{{(2)}}{\underline{{(2)}}}}=b_{\underline{{(1)}}}\ot b_{\underline{{(2)}}}\ot b_{\underline{{(3)}}}$.

\section{Some tools to compute the cocycles}\label{sec:tools}

Recall our setting from \S \ref{sec:intro}. Namely, $H$ is a semisimple Hopf algebra, $B\in\ydh$, $A=B\# H$ and we consider Hopf cocycles 
\[
\sigma\in Z^2_0(A)=\{\sigma\in Z^2(A)|\sigma_{|H\otimes H}=\eps\}\simeq Z^2(B)^H,
\]
so $\sigma=\sigma_B\#\eps$, $\sigma_B\in Z^2(B)^H$, see \eqref{eqn:sigma-intro}. We let $\mathbb{B}$ be a homogeneous basis of $B$, set $V\coloneqq B_1$ and fix $\{x_1,\dots,x_\theta\}\subset \mathbb{B}$ a basis of $V$.

In this section we collect some ideas that reduce the number of computations necessary to calculate the values
\[
\{\sigma(b,b'):b,b'\in\mathbb{B}\}.\]
We also investigate the orbit of $\sigma$ under the action of $U(A^\ast)$ in \S \ref{sec:orbit}.

\smallskip

The first result is straightforward.
\begin{lemma}\label{lem:cuentacociclo}
	Let $(E,\gamma)$ be a cleft object for $B$ in $H\mod$ and let $\sigma$ the associated cocycle as in \eqref{eqn:hopfcocycle}.
	For each $x\in \P(B), b\in B^+$, we have
	\begin{align*}
	\sigma(x,b)=\gamma(x_{(-1)}\cdot b_{(1)})\gamma^{-1}(x_{(0)}b_{(2)}).
	\end{align*}
\end{lemma}
\begin{proof}
	Indeed, we have that
	\begin{align*}
	\Delta(x\ot b)=x\ot b_{(1)}\ot 1\ot b_{(2)}+1\ot x_{(-1)}\cdot b_{(1)}\ot x_{(0)}\ot b_{(2)}.
	\end{align*}
	Hence we get
	\begin{align*}
	\sigma(x,b)&=\gamma(x)\gamma(b_{(1)})\gamma^{-1}(b_{(2)})+ \gamma(x_{(-1)}\cdot b_{(1)})\gamma^{-1}(x_{(0)}b_{(2)})\\
	&=	\gamma(x_{(-1)}\cdot b_{(1)})\gamma^{-1}(x_{(0)}b_{(2)}),
	\end{align*}
 by	\eqref{eqn:hopfcocycle}.
\end{proof}

Next lemma drastically reduces the number of computations, showing that the values $\{\sigma(b,b'):b,b'\in\mathbb{B}\}$ rely on the bialgebra structure of $B$ in $\ydh$ and the set:
\begin{align}\label{eqn:S-sigma}
S_\sigma\coloneqq\{\sigma(x_i,b): i\in \I_\theta; b\in \mathbb{B}\}.
\end{align}
Let $(c_{i,j}^{p,q})_{
	\begin{smallmatrix}
	i,j,\\ p,q
	\end{smallmatrix}
	\in\I_\theta}$ be the {\it matrix coefficients of the braiding} $c:V\ot V\to V\ot V$, so
\[
c(x_i\ot x_j)=\sum_{p,q\in \I_\theta} c_{i,j}^{p,q}x_p\ot x_q, \qquad i,j\in \I_\theta.
\]

\begin{lemma} \label{lem:descomposi}
	Let $a=xa',b\in B^+$, $x\in V$, $\deg a'>0$. If $a$ is a monomial, then
	\begin{align}\label{eqn:hop-cocycle-decomp-recursive}
	\begin{split}
	\sigma(a,b)=&\sigma(x,a'b)+\sigma(a',b_{\underline{{(1)}}})\sigma(x,b_{\underline{{(2)}}})+\sigma(a'_{\underline{{(1)}}},a'_{\underline{{(2)}}(-1)}\cdot b)\sigma(x,a'_{\underline{{(2)}}(0)})\\
	&+\sigma(a'_{\underline{{(1)}}},a'_{\underline{{(2)}}(-1)}\cdot b_{\underline{{(1)}}})\sigma(x,a'_{\underline{{(2)}}(0)}b_{\underline{{(2)}}})-\sigma(x,a'_{\underline{{(1)}}})\sigma(a'_{\underline{{(2)}}},b).
	\end{split}
	\end{align}
	In particular, for each $a,b\in\mathbb{B}^+$ there is a subset $S_\sigma^{a,b}\subseteq S_\sigma$ such that
	\begin{align}\label{eqn:hop-cocycle-decomp}
	\sigma(a,b)=\sum_{s_i\in S_\sigma^{a,b}}c_{s_1\dots s_\ell}\,s_1s_2\dots s_\ell.
	\end{align}
	for some $(c_{s_1\dots s_\ell})_{s_i\in S_\sigma^{a,b}}\in\mathbb{Z}[c_{i,j}^{p,q}]_{\begin{smallmatrix}
		i,j,\\ p,q
		\end{smallmatrix}
		\in\I_\theta}$.
\end{lemma}
\begin{proof}
	We proceed by induction on the degree $\deg(a)$ of a monomial $a\in B$. If $\deg(a)=1$, this is clear.
	Fix $n>1$ and assume that such decomposition exists for elements of degree $\leq n-1$.
	If $a\in B$ is of degree $n$, then there exists $x\in V$ such that $a=xa'$, where $\deg(a')=n-1$. Hence we have
	$\Delta(x)=x\ot 1+1\ot x$ and
	\begin{align*}
	\Delta(a')&=a'\ot 1+1\ot a'+a'_{\underline{{(1)}}}\ot a'_{\underline{{(2)}}}, & \Delta(b)&=b\ot 1+1\ot b+b_{\underline{{(1)}}}\ot b_{\underline{{(2)}}},
	\end{align*}
	{\it cf.}~\eqref{eqn:restr-comult}.
	In this case,  \eqref{eqn:hopfcocycle} becomes:
	\[
	\sigma(x_{(1)},a'_{(1)})\sigma(x_{(2)}a'_{(2)},b)=\sigma(a'_{(1)},b_{(1)})\sigma(x,a'_{(2)}b_{(2)}).
	\]
	Now, the left hand side can be developed further and we get
	\begin{align*} \sigma(x_{(1)}&,a'_{(1)})\sigma(x_{(2)}a'_{(2)},b)=\sigma(x,a')\sigma(1,b)+\sigma(x,1)\sigma(a',b)+\sigma(x,a'_{\underline{{(1)}}})\sigma(a'_{\underline{{(2)}}},b)\\
	&+\sigma(1,x_{(-1)}\cdot a')\sigma(x_{(0)},b)+\sigma(1,1)\sigma(xa',b)+\sigma(1,x_{(-1)}\cdot a'_{\underline{{(1)}}})\sigma(x_{(0)}a'_{\underline{{(2)}}},b)\\
	&=\sigma(a,b)+\sigma(x,a'_{\underline{{(1)}}})\sigma(a'_{\underline{{(2)}}},b),
	\end{align*}
	that is, $\sigma(a,b)=\sigma(x_{(1)},a'_{(1)})\sigma(x_{(2)}a'_{(2)},b)-\sigma(x,a'_{\underline{{(1)}}})\sigma(a'_{\underline{{(2)}}},b)$.
	
	Now we can develop the right hand side, and obtain:
	\begin{align*}
	\sigma(&a'_{(1)},b_{(1)})\sigma(x,a'_{(2)}b_{(2)})= \sigma(a',b)\sigma(x,1)+\sigma(a',1)\sigma(x,b)+\sigma(a',b_{\underline{{(1)}}})\sigma(x,b_{\underline{{(2)}}})\\
	&+\sigma(1,a'_{(-1)}\cdot b)\sigma(x,a'_{(0)})+\sigma(1,1)\sigma(x,a'b)+\sigma(1,a'_{(-1)}\cdot b_{\underline{{(1)}}})\sigma(x,a'_{(0)}b_{\underline{{(2)}}})\\
	&+\sigma(a'_{\underline{{(1)}}},{a'_{\underline{{(2)}}(-1)}}\cdot b)\sigma(x,a'_{\underline{{(2)}}(0)})+\sigma(a'_{\underline{{(1)}}},1)\sigma(x,a'_{\underline{{(2)}}}b)\\
	&+\sigma(a'_{\underline{{(1)}}},{a'_{\underline{{(2)}}(-1)}}\cdot b_{\underline{{(1)}}})\sigma(x,{a'_{\underline{{(2)}}(0)}}b_{\underline{{(2)}}})\\
	=& \, \sigma(x,a'b)+\sigma(a',b_{\underline{{(1)}}})\sigma(x,b_{\underline{{(2)}}})+\sigma(a'_{\underline{{(1)}}},{a'_{\underline{{(2)}}(-1)}}\cdot b)\sigma(x,a'_{\underline{{(2)}}(0)})\\
	&+\sigma(a'_{\underline{{(1)}}},a'_{\underline{{(2)}}(-1)}\cdot b_{\underline{{(1)}}})\sigma(x,a'_{\underline{{(2)}}(0)}b_{\underline{{(2)}}}).
	\end{align*}
	If $a=xa'$, then we obtain \eqref{eqn:hop-cocycle-decomp-recursive}.
	The conclusion is obtained by applying inductive hypothesis on each summand.
\end{proof}

\begin{definition}\label{def:hop-cocycle-decomp}
	The expression obtained in \eqref{eqn:hop-cocycle-decomp} via the recursive process \eqref{eqn:hop-cocycle-decomp-recursive} is the Hopf cocycle decomposition of $\sigma$ in terms of the set $S_{\sigma}$ from \eqref{eqn:S-sigma}.
\end{definition}

\begin{example}\label{exa:descompos-lowdregree}
	We exhibit explicitly the Hopf cocycle decomposition of a cocycle $\sigma$, in low degrees.
	
	Let $x,y,w,z\in V$ and $b\in \mathbb{B}^+$, then
	\begin{align*}\label{eq:descgrado2}
	\sigma(xy,b)=&\sigma(x,yb)+\sigma(x,b_{\underline{(2)}})\sigma(y,b_{\underline{(1)}}),\\
	\sigma(xyw,b)=&\sigma(x,ywb)-\sigma(x,y)\sigma(w,b)+\sigma(y,w_{(-1)}\cdot b)\sigma(x,w_{(0)})\\
	&+\sigma(y_{(-2)}\cdot w,y_{(-1)}\cdot b)\sigma(x,y_{(0)})+\sigma(yw,b_{\underline{(1)}})\sigma(x,b_{\underline{(2)}})\\
	&+\sigma(y,w_{(-1)}\cdot b_{\underline{(1)}})\sigma(x,w_{(0)}b_{\underline{(2)}})-\sigma(x,y_{(-1)}\cdot w)\sigma(y_{(0)},b)\\
	&+\sigma(y_{(-2)}\cdot w, y_{(-1)}\cdot b_{\underline{(1)}})\sigma(x,y_{(0)}b_{\underline{(2)}}),\\
	\sigma(xywz,b)=&\sigma(x,ywzb)-\sigma(x,y)\sigma(wz,b)-\sigma(x,yw)\sigma(z,b)\\
	&-\sigma(x,yw_{(-1)}\cdot z)\sigma(w_{(0)},b)-\sigma(x,y_{(-1)}\cdot(wz))\sigma(y_{(0)},b)&\\
	&-\sigma(x,y_{(-1)}\cdot w)\sigma(y_{(0)}z,b)-\sigma(x,(y_{(-1)}w_{(-1)})\cdot z)\sigma(y_{(0)}w_{(0)},b)
	\\
	&+\sigma(yw,z_{(-1)}\cdot b)\sigma(x,z_{(0)})+\sigma(y(w_{(-2)}\cdot z),w_{(-1)}\cdot b)\sigma(x,w_{(0)})\\
	&+\sigma(y,(w_{(-1)}z_{(-1)})\cdot b)\sigma(x,w_{(0)}z_{(0)})+\sigma(ywz,b_{\underline{(1)}})\sigma(x,b_{\underline{(2)}})\\
	&+\sigma(y_{(-2)}\cdot(wz),y_{(-1)}\cdot b)\sigma(x,y_{(0)})\\
	&+\sigma(y_{(-2)}\cdot(wz), y_{(-1)}\cdot b_{\underline{(1)}})\sigma(x,y_{(0)}b_{\underline{(2)}})\\
	&+\sigma(y_{(-2)}\cdot w,(y_{(-1)}z_{(-1)})\cdot b)\sigma(x,y_{(0)}z_{(0)})\\
	&+\sigma((y_{(-2)}w_{(-2)})\cdot z,(y_{(-1)}w_{(-1)})\cdot b)\sigma(x,y_{(0)}w_{(0)})
	\\
	&+\sigma(yw,z_{(-1)}\cdot b_{\underline{(1)}})\sigma(x,z_{(0)}b_{\underline{(2)}})\\
	&+\sigma(yw_{(-2)}\cdot z,w_{(-1)}\cdot b_{\underline{(1)}})\sigma(x, w_{(0)}b_{\underline{(2)}})
	\\
	&+\sigma(y,(w_{(-1)}z_{(-1)})\cdot b_{\underline{(1)}})\sigma(x,w_{(0)}z_{(0)}b_{\underline{(2)}})\\
	&+\sigma(y_{(-2)}\cdot w,(y_{(-1)}z_{(-1)})\cdot b_{\underline{(1)}})\sigma(x,y_{(0)}z_{(0)}b_{\underline{(2)}})
	\\
	&+\sigma((y_{(-2)}w_{(-2)})\cdot z, (y_{(-1)}w_{(-1)})\cdot b_{\underline{(1)}})\sigma(x,y_{(0)}w_{(0)}b_{\underline{(2)}}).
	\end{align*}
\end{example}

We have the following straightforward but useful consequence of the decomposition exhibited in Lemma \ref{lem:descomposi}.
\begin{corollary}\label{cor:primeralinea}
	Let $\sigma$ and $\sigma'$ be two Hopf cocycles in $B$.
	\begin{enumerate}[label=(\alph*)]
		\item If $\sigma(x_i,b)=\sigma'(x_i,b)$ for each $i\in\I_\theta$ and every $b\in\mathbb{B}$, then $\sigma=\sigma'$.
		\item If $\alpha\in B^\ast$ is convolution invertible and $\alpha\rightharpoonup\sigma(x_i,b)=\sigma'(x_i,b)$, for each $i\in\I_\theta$ and every $b\in\mathbb{B}$, then $\alpha\rightharpoonup\sigma=\sigma'$.\hfill\qed
	\end{enumerate}
\end{corollary}

\subsection{Some remarks on $\alpha\rightharpoonup\sigma$}\label{sec:orbit}

Let $\sigma\in Z^2_0(A)$, $\alpha\in U(A^*)$, so that $\sigma'\coloneqq \alpha\rightharpoonup\sigma$ as in \eqref{eqn:hopfcohomologous} is a Hopf cocycle cohomologous to $\sigma$. Moreover, we assume that $\sigma'\in Z^2_0(A)$.
Recall that we may assume $\alpha(1)=1$.

\begin{remark}
	$\gamma_\alpha\coloneqq \alpha\ast\gamma\colon A\to E$ is another section, with associated cocycle $\alpha\rightharpoonup\sigma$.
\end{remark}

\begin{lemma}
	If $\alpha\ra\sigma=\sigma'$, then $\alpha_{|H}\in\Alg(H,\k)$. Hence $\alpha_{|H}^{-1}=(\alpha\circ\Ss)_{|H}$.
\end{lemma}
\pf
The cocycle $\sigma_\alpha$ associated to $\gamma_\alpha$ is $\sigma'$, which satisfies $\sigma'(h,k)=\eps(hk)$. But
\[
\eps(hk)=\sigma_\alpha(h,k)=\alpha(h_{(1)})\alpha(k_{(1)})\alpha^{-1}(h_{(2)}k_{(2)}),
\]
which shows that $\alpha(hk)=\alpha(h)\alpha(k)$, by uniqueness of the convolution inverse.
\epf

\begin{lemma}\label{lem:alpha-lineal}
	We have $\alpha(xh)=\alpha(x)\alpha(h)$ and $\alpha(hx)=\alpha(x)\alpha(h)$.
	
	In particular, $\alpha_{|B}:B\to \k$ is $H$-linear. That is,
	\[
	\alpha=\alpha'\#\varphi, \quad\text{for } \alpha'\in\Hom_H(B,\k) \text{ and } \varphi\in\Alg(H,\k).
	\]
\end{lemma}
\pf
Notice that
\begin{align}\label{eqn:alpha2}
\eps(xh)=\alpha\ra\sigma(x,h)=\alpha(x_{(1)})\alpha(h_{(1)})\alpha^{-1}(x_{(2)}h_{(2)}).
\end{align}
Thus $\psi(xh)=\alpha(x)\alpha(h)$ is a left convolution inverse to $\alpha^{-1}$; hence $\psi=\alpha$.
The other equality is similar, considering $\alpha\ra\sigma(h,x)$.

Now, $\alpha_{|B}(h\cdot x)=\alpha(h_{(1)}x \Ss(h_{2}))=\alpha(h_{(1)})\alpha(x)\alpha(\Ss(h_{2}))=\eps(h)\alpha(x)$.
\epf

\begin{lemma}
	We can assume, without loss of generality, that $\alpha_{|H}=\eps$.
\end{lemma}
\pf
Set $\varphi\coloneqq\alpha_{|H}$; if $\varphi\neq \eps$, consider $\tau=\eps\#\varphi^{-1}$ and set $\tilde{\alpha}\coloneqq\alpha\ast\tau$. On the one hand, $\tilde\alpha_{|H}=\eps$, and $\tilde{\alpha}(xh)=\alpha(x)\eps(h)$, $\tilde{\alpha}^{-1}(xh)=\alpha^{-1}(x)\eps(h)$. Now, observe that
\begin{align*}
\tilde\alpha&\ra\sigma(xh,yk)=\tilde\alpha(x_{(1)}h_{(1)})\tilde\alpha(y_{(1)}k_{(1)})\sigma(x_{(2)}h_{(2)}, y_{(2)}k_{(2)})\tilde\alpha^{-1}(x_{(3)}h_{(3)} y_{(3)}k_{(3)})\\
&=\alpha(x_{(1)})\alpha(y_{(1)})\sigma(x_{(2)}h_{(1)}, y_{(2)})\tilde\alpha^{-1}(x_{(3)}h_{(2)}\cdot y_{(3)})\eps(k)\\
&=\alpha(x_{(1)})\alpha(y_{(1)})\sigma(x_{(2)}h_{(1)}, y_{(2)})\alpha^{-1}(x_{(3)}h_{(2)}\cdot y_{(3)})\eps(k)\\
&=\alpha(x_{(1)})\alpha(h_{(1)}\cdot y_{(1)})\sigma(x_{(2)}, h_{(2)}\cdot y_{(2)})\alpha^{-1}(x_{(3)}h_{(3)}\cdot y_{(3)})\eps(k)\\
&=\sigma'(x,h\cdot y)\eps(k)=\sigma'(xh,yk).
\end{align*}
Hence we can replace $\alpha$ by $\tilde{\alpha}$.
\epf

\begin{proposition}\label{pro:alpha-extendido}
	Let $A=B\#H$. Fix $\sigma=\sigma_B\#\eps, \sigma'=\sigma'_B\#\eps\in Z^2(A)$.
	\begin{enumerate}[label=(\alph*)]
		\item If $\alpha'\colon B\to\k$ is convolution invertible, $H$-linear, and $\alpha'\ra\sigma_B=\sigma'_B$, then $\alpha \ra \sigma=\sigma'$, for $\alpha=\alpha'\#\eps$.
		\item If $\alpha\colon A\to\k$ is convolution invertible is such that $\alpha\ra\sigma=\sigma'$, then we have $\alpha_{|B}\ra \sigma_B=\sigma'_B$.
	\end{enumerate}
\end{proposition}
\pf
$(a)$ We have that $\alpha'\ra\sigma_B=\sigma'_B$, that is for $x,y\in B$,
\[
\alpha'(x_{(1)})\alpha'(x_{(3)(-2)}x_{(2)(-1)}\cdot y_{(1)})\sigma_B(x_{(2)(0)},x_{(3)(-1)}y_{(2)})\alpha'{}^{-1}(x_{(3)(0)}y_{(3)})=\sigma'_B(x,y).
\]
Here $x_{(1)}\ot x_{(2)}=\Delta_B(x)$ stands for the coproduct in $B$. Recall that the coproduct
$\Delta(xh)=x_{(1)}x_{(2)(-1)}h_{(1)}\ot x_{(2)(0)}h_{(2)}$ in $A$ from \eqref{eqn:comultBH} so
\begin{align*}
\Delta^{(2)}(xh)&=x_{(1)}x_{(2)(-1)}x_{(3)(-2)}h_{(1)} \ot  x_{(2)(0)}x_{(3)(-1)}h_{(2)}\ot x_{(3)(0)}h_{(3)}.
\end{align*}
Now,  $\alpha\rightharpoonup \sigma(xh,yk)$ gives, using that that $\alpha'$ is $H$-linear:
\begin{align*}
\alpha(x_{(1)}&x_{(2)(-1)}x_{(3)(-2)} h_{(1)})\alpha(y_{(1)}y_{(2)(-1)}y_{(3)(-2)} k_{(1)})\\
&\hspace*{1.5cm} \times\sigma(x_{(2)(0)}x_{(3)(-1)}h_{(2)},y_{(2)(0)}y_{(3)(-1)}k_{(2)})\alpha^{-1}(x_{(3)(0)}h_{(3)}y_{(3)(0)}k_{(3)})\\
&=\alpha'(x_{(1)})\alpha'(y_{(1)})\sigma_B(x_{(2)},x_{(3)(-1)}h_{(1)}\cdot y_{(2)}) \alpha'{}^{-1}(x_{(3)(0)}h_{(2)}y_{(3)})\eps(k)\\
&=\alpha'(x_{(1)})\alpha'(x_{(3)(-2)}x_{(2)(-1)}h_{(1)}\cdot y_{(1)})\\
&\hspace*{1.5cm} \times\sigma_B(x_{(2)(0)},x_{(3)(-1)}h_{(2)}y_{(2)})\alpha'{}^{-1}(x_{(3)(0)}h_{(3)}y_{(3)})\eps(k)\\
&=\sigma'_B(x,h\cdot y)\eps(k)=\sigma'(xh,yk).
\end{align*}

$(b)$ Follows as $(a)$, using that $\alpha'$ is $H$-linear by Lemma \ref{lem:alpha-lineal} and the compatibility between $\Delta_B(x)$ and $\Delta(x\#1)$, see \eqref{eqn:comultBH}.
\epf

Next lemma will also allow us to reduce computations when dealing with Hopf cocycles arising as exponentials of Hochschild cocycles.

\begin{lemma}\label{lem:cuentaalpha}
Let $\tau,\sigma\in Z^2(B)^H$ and let $\alpha\in B^\ast$ be convolution invertible.
	\begin{enumerate}[label=(\alph*)]
		\item For any $x\in V$, $b\in\mathbb{B}$, we have that
		\[
		\alpha\rightharpoonup \tau (x,b)=\alpha(x_{(-1)}\cdot b_{(1)})\tau(x_{(0)}, b_{(2)})\alpha^{-1}(b_{(3)})+\alpha(x_{(-1)}\cdot b_{(1)})\alpha^{-1}(x_{(0)}b_{(2)}).
		\]
		\item Assume that $\alpha\rightharpoonup \tau=\sigma$. Then
		for any $x,y\in V$, we have that
\begin{align}\label{eqn:alpha-grado1}
\alpha(x)\alpha(y)-\alpha(xy)=\sigma(x,y)-\tau(x,y).
\end{align}
	\end{enumerate}
\end{lemma}
\begin{proof}
(1)	Let $x\in V$ and $b\in \mathbb{B}$, then
	\begin{align*}
	(\id\ot \Delta)\Delta(x\ot b)=& \, x\ot b_{(1)}\ot 1\ot b_{(2)}\ot 1\ot b_{(3)}\\
&+1\ot x_{(-1)}\cdot b_{(1)}\ot x_{(0)}\ot b_{(2)}\ot 1\ot b_{(3)}\\
&+1\ot x_{(-2)}\cdot b_{(1)}\ot1\ot x_{(-1)}\cdot b_{(2)}\ot x_{(0)}\ot b_{(3)}.
	\end{align*}
	When we act by $\alpha$, we get
	\begin{align*}
	\alpha\ra\tau(x,b)=& \,\alpha(x)\alpha(b_{(1)})\tau(1,b_{(2)})\alpha^{-1}(b_{(3)})+\alpha(x_{(-1)}\cdot b_{(1)})\tau(x_{(0)}, b_{(2)})\alpha^{-1}(b_{(3)})\\
&+\alpha(x_{(-2)}\cdot b_{(1)})\tau(1, x_{(-1)}\cdot b_{(2)})\alpha^{-1}(x_{(0)}b_{(3)})\\
=& \,\alpha(x)\alpha(b_{(1)})\eps(b_{(2)})\alpha^{-1}(b_{(2)})+\alpha(x_{(-1)}\cdot b_{(1)})\tau(x_{(0)}, b_{(2)})\alpha^{-1}(b_{(3)})\\
&+\alpha(x_{(-1)}\cdot b_{(1)})\alpha^{-1}(x_{(0)}b_{(2)})\\
=&\,\alpha(x_{(-1)}\cdot b_{(1)})\tau(x_{(0)}, b_{(2)})\alpha^{-1}(b_{(3)})+\alpha(x_{(-1)}\cdot b_{(1)})\alpha^{-1}(x_{(0)}b_{(2)}).
	\end{align*}
(2) Use the formula in (1) to compute $\alpha\rightharpoonup \tau(x,y)=\sigma(x,y)$, then
\[\alpha(x_{(-1)}\cdot y)\alpha^{-1}(x_{(0)})+\alpha^{-1}(xy)=\sigma(x,y)-\tau(x,y).\]
Now, as
\begin{align}\label{eqn:alpha-1xy}
\alpha^{-1}(xy)=\alpha(x)\alpha(y)-\alpha(x_{(-1)}\cdot y)\alpha^{-1}(x_{(0)})-\alpha(xy),
\end{align}
we obtain \eqref{eqn:alpha-grado1}.
\end{proof}

\section{Invariant Hochschild cohomology}\label{sec:hoch}

We let $B$, $H$, $A=B\# H$ be as in \S \ref{sec:setting}. Recall that $\{x_1,\dots,x_\theta\}$ denotes a linear basis of $V\coloneqq B_1$, contained in an homogeneous basis $\mathbb{B}$ of $B$.

Let $M$ be an $A$-bimodule; hence a $B$-bimodule by restriction.
In this section we focus on the Hochschild cohomology $\H^{\bullet}$ of $A$ with coefficients on $M$ and the relation with $\H^{\bullet}(B,M)$.

In particular, we present a proof of a Theorem of \c{S}tefan in \cite{St}, which shows that there is an action of $H$ on $\H^{\bullet}(B,M)$ in such a way that $\H^{\bullet}(A,M)\simeq \H^{\bullet}(B,M)^H$. We give an explicit description of this action in our context.

As well, we will be particularly interested in the case $M=\k$.
We write $\Z^{\bullet}(A,M)$ for the space of cocycles and  $\B^{\bullet}(A,M)$ for the cobordisms. When $M=\k$, we may simply write $\H^{\bullet}(A), \Z^{\bullet}(A), \B^{\bullet}(A)$. We denote by $\Z^{\bullet}(B)^H\subset \Z^{\bullet}(B)$ the subset of $H$-invariant cocycles. The subset $\Z^{2}(B)^H$ will be of big relevance in our work.

We refer the reader to \cite{book-cohom} for details on Hochschild cohomology.

\subsection{Hochschild 2-cocycles with coefficients in $\k$}

Let $B^+=\bigoplus_{n > 0} B_n$.
We begin with an alternative description of the set $\Z^2(B,\k)$.

\begin{lemma}\label{lem:hoch}
Let $\eta:B\ot B\to \k$ be a linear map.	The following are equivalent.
	\begin{enumerate}[label=(\alph*)]
\item  $\eta\in \Z^2(B,\k)$.
\item $\eta:B\ot B\to \k$ satisfies, for $a,c\in B^+$,
\begin{align}
\label{cond1}
\eta(a,1)&=\eta(1,c)=0,\\
\label{cond2a}
\eta(a,bc)&=\eta(ab,c), \quad b\in B.
\end{align}
	\end{enumerate}
Similarly, $\eta\in \B^2(B,\k)$ if and only if \eqref{cond1} holds and there is a linear map $f\colon B\to \k$ such
that $\eta(1,1)=f(1)$ and
\[
\eta(a,b)=-f(ab), \quad a,b\in B^+.
\]
\end{lemma}
\pf
Recall that $\eta\in\Z^2(B,\k)$ if and only if, for $a,b,c\in B$:
\begin{equation}\label{hoch}
\eps(a)\eta(b,c)+\eta(a,bc)=\eta(ab,c)+\eta(a,b)\eps(c).
\end{equation}
Now, if $a,c\in B^+$, then this becomes \eqref{cond2a}.
On the other hand, if $a=\eps(a)\in B_0=\k$ and $c\in B^+$ (or viceversa), then this is
\begin{align*}
\eps(a)\eta(b,c)+\eps(a)\eta(1,bc)=\eps(a)\eta(b,c)
\Rightarrow \eta(1,bc)=0 \qquad (\text{or }\eta(ab,1)=0).
\end{align*}
If $a,c\in B_0=\k$ then  \eqref{hoch} is tautological.

Recall that $\eta\in\B^2(B,\k)$ if and only if there is $f\colon B\to \k$ in such a way that
\[
\eta(a,b)=\eps(a)f(b)-f(ab)+f(a)\eps(b).
\]
This is in turn equivalent to
$\eta(a,b)=-f(ab)$, $a,b\in B^+$.
\epf

We compute an easy example in \eqref{eqn:etabi} to fix some notation. We will apply Lemma \ref{lem:hoch} to write down in \S \ref{sec:hoch-a2} two more non-trivial examples in \eqref{eqn:xi-121} and \eqref{eqn:xi-212}.

\begin{lemma}\label{exa:xi_ij}
For each $i,j\in I_\theta$, let $\xi_j^i\colon B\ot B\to \k$ be the linear map defined by
\begin{align}\label{eqn:etabi}
\xi_j^i(b,b')=\delta_{b,x_j}\delta_{x_i,b'}, \ b,b'\in\mathbb{B}.
\end{align}
Then $\xi_j^i\in \Z^2(B,\k)$.\hfill\qed
\end{lemma}

\begin{remark}\label{rem:hoch}
{\it In the sense} of Lemma \ref{lem:descomposi}, \eqref{cond2a} states that $\eta$ is determined by the values in the set $\{\eta(b,x_i)|b\in\mathbb{B},i\in\I_\theta\}\cup\{\eta(1,1)\}$.
\end{remark}
%
%
%
%

\subsection{\c{S}tefan's Theorem revisited}
The following result, in general terms is due to \c{S}tefan \cite{St}, and it is of central importance in our work.
It is in fact a compendium of several results in {\it loc.cit.}, which we have adapted to our setting, and this allows us to find explicit formulas that are implicit in Stefan's work. It is our belief that it is in this shape that might evoke more interest in the scope of this article. In particular, we state it in terms of left actions because of this idea.

\begin{theorem}\label{thm:stefan}
Let $H$ be a Hopf algebra with bijective antipode and let $B\in\ydh$. Let $A=B\# H$ and let $M$ be an $A$-bimodule.
	
Then there is a left $H$-action $\ra$ on $\H^q(B,M)$, $q\geq 0$.
\begin{itemize}
	\item For $q=0$, this is
	\begin{align}\label{eqn:H0-action}
	h\ra m=h_{(1)}\cdot m\cdot \Ss(h_{(2)}),
	\end{align}
	for $h\in H$ and $m\in \H^0(B,M)=\{m:b\cdot m=m\cdot b, \forall\,b\in B\}\subseteq M$.
	\item For $q>0$, this is
	\begin{align}\label{eqn:Hq-action}
	(h\ra f)({\bf x})=h_{(1)}\cdot f(\Ss(h_{(2)})\cdot {\bf x})\cdot \Ss(h_{(3)}),
	\end{align}
	for $h\in H$ and $f\in \Z^q(B,M)\subset\Hom_\k(B^{\ot\,q}, M)$, ${\bf x}=(x_1,\cdots, x_q)\in B^{\ot\,q}$.
\end{itemize}
	 If $H$ is semisimple, then
	 \begin{equation}\label{eqn:invariance}
	 	\H^n(A,\k)\simeq \H^n(B,\k)^H.
	 \end{equation}
\end{theorem}

\pf
Following \cite{St} we work, equivalently, with a right action.
%
We update our equations and need to show that there is a right $H$-action $\la$ on $\H^q(B,M)$ so that
\begin{itemize}
	\item For $q=0$, this is
	\begin{align}\tag{\ref{eqn:H0-action}'}\label{eqn:H0-action-right}
   m\la h=\Ss(h_{(1)})\cdot m \cdot h_{(2)}, \quad m\in \H^0(B,M), h\in H;
	\end{align}
	\item For $q>0$, this is
	\begin{align}\tag{\ref{eqn:Hq-action}'}\label{eqn:Hq-action-right}
	(f\la h)({\bf x})=\Ss(h_{(1)})\cdot f(h_{(2)}\cdot {\bf x})\cdot h_{(3)}, \quad f\in \Z^q(B,M), {\bf x}\in B^{\ot\,q}.
	\end{align}
Observe that the left action $h\cdot b$ can also be interpreted as $b\cdot \Ss^{-1}(h)$.
\end{itemize}

First, recall that $A$ is a right $H$-Galois object with coaction $A\to A\ot H$ given by
$b\# h\mapsto b\# h_{(1)}\ot h_{(2)}$, with  $A^{\co H}\simeq B$ by \cite[Theorem 2.2.7]{Sch}.

We let $\can\colon A\ot_B A\to A\ot H$ be the canonical isomorphism; in particular,
\[
\sum \ell_i(h)\ot r_i(h)\coloneqq \can^{-1}(1\ot h)=\Ss(h_{(1)})\ot h_{(2)}.
\]

For $M\in A\bimod$, we write $\cdot$ both for the left and right action.
By \cite{St}, the $H$-action $\la$ on $H^0(B,M)=M^{B}$ is given by
$m\la h=\sum \ell_i(h)\cdot m\cdot  r_i(h)$, which in this case translates to \eqref{eqn:H0-action-right}:
\[
m\la h=\Ss(h_{(1)})\cdot m\cdot h_{(2)}.
\]
For left actions, this is \eqref{eqn:H0-action}.

Assume this action \eqref{eqn:H0-action-right} has been extended to $\H^q(B,M)$ (here we admit $q=0$).

Next paragraph is almost {\it verbatim} from the proof of \cite[Proposition 2.4]{St}. We refer the reader to loc.cit.~and references therein for fuller details.

If $R\colon A\bimod\to B\bimod$ is the functor given by restriction of scalars, then $\H^{\bullet}(A,-)\circ R$ is a cohomological and coeffaceable $\delta$-functor. Hence the morphisms $\rho^h_0\colon \H(A,-)\circ R\to \H(A,-)\circ R$, $h\in H$, given by $\rho^h_0(M)(m)=m\la h$ can be extended univocally to morphisms of $\delta$-functors $\rho^h_{\bullet}\colon \H^{\bullet}(A,-)\circ R\to \H^{\bullet}(A,-)\circ R$.

We will describe this action explicitly in our context. For the content of this proof, we  shall write, for shortness,
$H^\bullet(A,-)\circ R(M)=H^\bullet(M)$, $M\in\A\bimod$. We denote by $d^q_{M}\colon \Hom_\k(B^{\ot q},M)\to \Hom_\k(B^{\ot q+1},M)$ the differential.

Let us fix a short exact sequence
\[
0\to  M'\stackrel{\iota}{\longrightarrow} M \stackrel{\pi}{\longrightarrow} M''\to 0
\]
in $A\bimod$; we write $M\ni m\twoheadrightarrow\bar{m}\in M''$ for the projection.

First, the fact that $\rho^h_{\bullet}$ is a morphism of $\delta$-functors establishes that
\[
\rho^h_{q+1}(\delta(\H^q(M'')))=\delta(\rho^h_{q}(\H^q(M''))).
\]
In other words this is the recipe to extend the action, as it gives
\[
\delta([f])\la h=\delta([f\la h]), \qquad f\colon B^{\ot q}\to M''\in\Z^q(M'').
\]
We shall write, wlog, $\delta(f)\colon B^{\ot q}\to M'$ for an element in $\Z^{q+1}(M')$ such that $[\delta(f)]=\delta([f])$.
Now, the morphism $\delta\colon \H^q(M'')\to \H^{q+1}(M')$ is the usual connecting homomorphism for the long exact sequence in cohomology:
\[
\dots \to  \H^q(M) \stackrel{\H^q(\pi)}{\longrightarrow} \H^q(M'')\stackrel{\delta}{\longrightarrow}\H^{q+1}(M') \stackrel{\H^{q+1}(\iota)}{\longrightarrow} \H^{q+1}(M) \to \dots
\]
Explicitly, for $f\in\Z^q(M'')=\ker d^q_{M''}$, let $\tilde{f}\colon B^{\ot q}\to M$ be such that $\pi(\tilde{f})=f$, in the sense that $f(b_1,\dots,b_q)(\pi(m))=\tilde{f}(b_1,\dots,b_q)(m)$. Now we consider $d^q_M({\tilde{f}})\in \Z^{q+1}(M'')$ and see that $\H^{q+1}(\pi)([d^q_M({\tilde{f}})])=d^q_{M''}(\H^{q}(\pi)(\tilde{f}))=d^q_{M''}(f)=0$.
That is $[d^q_M({\tilde{f}})]\in\ker \H^{q+1}(\pi)=\img  \H^{q+1}(\iota)$ and we can consider $[d^q_M({\tilde{f}})]$ as an element $\delta([f])\in \H^{q+1}(M')$. By abuse of notation we shall write $\delta(f)=df$.

In our setting, if ${\bf b}=(b_1,\dots,b_{q+1})$
\begin{align*}
(\delta(f)\la h)({\bf b})=d(f \la h)({\bf b})
\end{align*}
If $q=0$, $m\in M, b\in B$ , then $dm(b)=b\cdot m-m\cdot b$.
Also, recall the commutations \eqref{eqn:hb}. The connection reads as
\begin{align*}
(\delta(m)\la h)(b)&=d(m \la h)(b)=(d(\Ss(h_{(1)})\cdot m \cdot h_{(2)})(b))\\
&=b\cdot \Ss(h_{(1)})\cdot m\cdot h_{(2)} -\Ss(h_{(1)})\cdot m \cdot h_{(2)}\cdot b\\
&=(b\Ss(h_{(1)}))\cdot m\cdot h_{(2)} -\Ss(h_{(1)})\cdot m \cdot (h_{(2)}b)\\
&=\Ss(h_{(1)})(h_{(2)}\cdot b)\cdot m \cdot h_{(3)}-\Ss(h_{(1)})\cdot m \cdot (h_{(2)}\cdot b)h_{(3)}\\
&=\Ss(h_{(1)})\cdot (h_{(2)}\cdot b\cdot m - m\cdot h_{(2)}\cdot b)\cdot h_{(3)}\\
&=\Ss(h_{(1)})\cdot dm(h_{(2)}\cdot b)\cdot h_{(3)},
\end{align*}
so $dm\la h$ is as in \eqref{eqn:H0-action-right}.

Now, we turn to $q>0$. If ${\bf b}=b_1\ot \cdots \ot b_{q+1}\in B^{\ot\,q+1}$, then recall that  we write ${\bf b}=b_1\ot {\bf b}^2={\bf b}^q\ot b_{q+1}$ and ${\bf b}_{i\,i+1}=b_1\ot \cdots \ot b_ib_{i+1}\ot \cdots\ot b_{q+1}\in B^{\ot\,q}$, $i\in\I_q$.

Hence, we obtain, with commutations \eqref{eqn:hb} as above:
\begin{align*}
d(f \la h)&({\bf b})=b_1\cdot (f \la h)({\bf b}^2)-(f \la h)({\bf b}_{12})+(f \la h)({\bf b}_{23})\\
&\qquad +\cdots +(-1)^{n-1}(f \la h)({\bf b}_{q-1q})+(-1)^n(f \la h)({\bf b}^q)\cdot b_{q+1}\\
&=b_1\cdot \Ss(h_{(1)})\cdot f(h_{(2)}\cdot {\bf b}^2) \cdot h_{(3)}-\Ss(h_{(1)})\cdot f(h_{(2)}\cdot {\bf b}_{12})h_{(3)}\\
&\qquad +\cdots +(-1)^n\Ss(h_{(1)})\cdot f (h_{(2)}\cdot {\bf b}^q)\cdot h_{(3)}\cdot b_{q+1}\\
&=\Ss(h_{(1)} ) \cdot h_{(2)}\cdot b_1 \cdot f(h_{(3)}\cdot {\bf b}^2) \cdot h_{(3)}-\Ss(h_{(1)})\cdot f(h_{(2)}\cdot {\bf b}_{12})h_{(3)}\\
&\qquad +\cdots +(-1)^n\Ss(h_{(1)})\cdot f (h_{(2)}\cdot {\bf b}^q)\cdot h_{(3)}\cdot b_{q+1}\cdot h_{(4)}\\
&=\Ss(h_{(1)} ) \cdot (h_{(2)}\cdot {\bf b})_1 \cdot f((h_{(2)}\cdot {\bf b})^2) \cdot h_{(3)}-\Ss(h_{(1)})\cdot f((h_{(2)}\cdot {\bf b})_{12})h_{(3)}\\
&\qquad +\cdots +(-1)^n\Ss(h_{(1)})\cdot f ((h_{(2)}\cdot {\bf b})^q)\cdot (h_{(2)}\cdot {\bf b})_{q+1}\cdot h_{(3)}\\
&=\Ss(h_{(1)})df(h_{(2)}\cdot {\bf b})h_{(3)},
\end{align*}
which shows \eqref{eqn:Hq-action-right}.

By  uniqueness of the extension; this shows that  \eqref{eqn:Hq-action-right} is indeed the action announced on \cite[Proposition 2.4]{St}; which implies \eqref{eqn:Hq-action} for the case of  left actions.

The identity \eqref{eqn:invariance} is \cite[(3.6.1) in p.229]{St}.
\epf

\subsection{Invariant Hochschild 2-cocycles with coefficients in $\k$}

In this setting, the action $\ra$ in \eqref{eqn:Hq-action} from Theorem \ref{thm:stefan} becomes
\[
h\ra \eta(x,y)= \eta(\Ss(h_{(2)})\cdot x,\Ss(h_{(1)})\cdot y).
\]
Since the antipode $\Ss$ is bijective, a cocycle $\eta\in \Z^2(B)$ is $H$-invariant when
\begin{align}\label{eqn:eta-invariant}
\eta(h_{(1)}\cdot a,h_{(2)}\cdot b)=\eps(h)\eta(a,b), \ a,b\in B.
\end{align}

Similarly, if $\eta\in Z^2(B)^H$, then $\eta$ extends to $\eta\in \Z^2(B\# H)$ via the formula
\begin{align}\label{eqn:hoch-extension}
\eta (a\#h,y\#k)=\eta(a,h\cdot b)\eps(k), \ a,b\in B, h,k\in H.
\end{align}

Conversely, if $\eta\in Z^2(B\# H)$ then the restriction
\[
\eta_{|B}\coloneqq \eta_{|B\otimes B}\colon B\otimes B\to \k
\]
defines a cocycle $\eta_{|B}\in Z^2(B)$; this is not necessarily $H$-invariant. If $\int\in H$ is an integral with $\eps(\int)=1$, then we can project $\eta\in Z^2(B)$ onto $\eta^{\int}\in Z^2(B)^H$ by
\[	
\eta^{\int}(a,b)=\eta(\smallint{}_{(1)}\cdot a,\smallint{}_{(2)}\cdot b), \qquad a,b\in B.
\]
%

\section{Hopfschild cohomology}\label{sec:hopfschild}

Let $(A,m,\Delta)$ be a Hopf algebra and let $\eta\in \Z^2(A,\k)$ be a  Hochschild 2-cocycle with coefficients in $\k$.
Under certain conditions, it is possible to consider the exponential
\[
e^\eta\coloneqq \sum_{k\geq 0} \frac{1}{k!}\eta^{\ast k}\colon A\ot A\to \k,
\]
which becomes a convolution invertible map with inverse $e^{-\eta}$. For instance, this holds when $\eta_{|A\ot A_0+A_0\ot A}=0$.

We show that this map is always well-defined in our setting in \S \ref{sec:expo-and-inv}.

Moreover, there are occasions in which this exponential map is a Hopf 2-cocycle. 
A sufficient condition is given by the $\ast$-commutations in $\Hom(A^{\ot 3},\k)$:
\begin{align}
\label{eqn:conm1}
[\eta(\id\ot m), \eps\ot \eta]_{\ast}&=0,\\
\label{eqn:conm2} [\eta(m\ot \id), \eta\ot \eps]_{\ast}&=0.
\end{align}
However, we observe the following:
\begin{remark}
Conditions  \eqref{eqn:conm1} and \eqref{eqn:conm2} are not necessary, see Lemma \ref{lem:converse} in \S \ref{sec:a2}.
\end{remark}
On the other hand, under some hypothesis, these equations are equivalent to each other, see \cite{GaM}. We investigate some of these instances in \S \ref{sec:when}.

The sufficiency of \eqref{eqn:conm1} and \eqref{eqn:conm2} is clear: for $\sigma=e^\eta$, \eqref{eqn:hopfcocycle-cordfree} becomes
\begin{align*}
(\sigma\ot\eps)\ast \sigma(m\ot\id)&=e^{\eta\ot\eps}\ast e^{\eta(m\ot\id)}\stackrel{\eqref{eqn:conm2}}{=}e^{\eta\ot\eps+\eta(m\ot\id)}
\stackrel{\eqref{hoch}}{=} e^{\eps\ot\eta+\eta(\id\ot\, m)}\\
&\stackrel{\eqref{eqn:conm1}}{=}
e^{\eps\ot\eta}\ast e^{\eta(\id\ot\, m)}=(\eps\ot\sigma)\ast \sigma(\id\ot\, m).
\end{align*}
This is the content of \cite[Lemma 4.1]{GrM}.

Notice that, in terms of elements, \eqref{eqn:conm1} and \eqref{eqn:conm2} are, for $x,y,z\in A$:
\begin{align}\label{eqn:conds-in-elements}
\begin{split}
\eta(x,y_{(1)}z_{(1)})\eta(y_{(2)},z_{(2)})&=\eta(y_{(1)},z_{(1)})\eta(x,y_{(2)}z_{(2)}), \\ \eta(x_{(1)}y_{(1)},z)\eta(x_{(2)},y_{(2)})&=\eta(x_{(1)},y_{(1)})\eta(x_{(2)}y_{(2)},z).
\end{split}
\end{align}

Based on our examples, we observe that Hopf cocycle deformations are generically determined by cocycles which do not arise as exponentials. Hence we introduce the following concept, to name and study this phenomenon.
\begin{definition}
We say that a Hopf 2-cocycle $\sigma\in Z^2(A)$ is pure if it is not cohomologous to an exponential $e^\eta$ of a Hochschild 2-cocycle $\eta\in \Z^2(A,\k)$.
\end{definition}

We shall speak of pure deformations and exponential deformations, as expected.

\subsection{Exponential and invariance}\label{sec:expo-and-inv}

Let $H,B\in\ydh$, $A=B\# H$ be as in \S \ref{sec:setting}.

We start with some simple though meaningful observations in Lemma \ref{lem:eta11}.
Let $\eta\in \Z^2(B,\k)^H$ be an $H$-invariant Hochschild 2-cocycle; and set $\eta\in \Z^2(A,\k)$ the extension as in \eqref{eqn:hoch-extension}. Notice that
\[
\eta(x,h)=\eta(x,1)\eps(h)=0, \qquad \eta(h,x)=\eta(1,h\cdot x)=0,
\]
so $\eta_{|A^+\ot A_0+A_0\ot A^+}=0$. Hence if $\eta_{|A_0\ot A_0}=0$, then the exponential
$\sigma\coloneqq e^\eta$ is well-defined.
Moreover, as in our case $\eta(h,k)=\eps(hk)\eta(1,1)$, $h,k\in H=A_0$, it is enough to require
\begin{align}\label{eqn:eta11}
\eta(1,1)=0.
\end{align}

Next lemma shows that we can always assume \eqref{eqn:eta11}, without loss of generality.

\begin{lemma}\label{lem:eta11}
Let $\eta\in \Z^2(B,\k)^H$ and set $
\eta'=\eta - \eta(1,1)\eps$. Then
\begin{enumerate}[label=(\alph*)]
	\item $\eta'\in \Z^2(B,\k)$ and $[\eta]=[\eta']$ in $\H^2(B,\k)$.
	\item  $\eta'(1,1)=0$ and $\eta'=\eta$ in $B^+\ot B+B\ot B^+$.
	\item $e^{\eta'}\colon B\ot B\to \k$ is a normalized convolution invertible linear map and
\[
e^{\eta}=e^{\eta(1,1)}e^{\eta'}.
\]
In particular, 	$e^\eta\colon B\ot B\to \k$ is a convolution invertible linear map.
\item $e^\eta$ is a Hopf 2-cocycle if and only if $e^{\eta'}$ is a (normalized) Hopf 2-cocycle.
\end{enumerate}
\end{lemma}
\pf
Let $\theta=\eta(1,1)\eps$, then $\theta\in \B^2(B,\k)$ and  $(a)$ holds. Also, $e^\theta=e^{\eta(1,1)}\eps$. $(b)$ is clear and it implies that $e^{\eta'}\colon B\ot B\to \k$ is well-defined normalized convolution invertible linear map. As $\eta'\ast\theta=\theta\ast\eta'$, then $e^{\theta}\ast e^{\eta'}=e^{\eta}$, which proves $(c)$. $(d)$ is straightforward, cf.~\eqref{eqn:normalized}.
\epf

\begin{remark}\label{rem:minimo}
	If $\eta(1,1)=0$, then
\[
e^{\eta}(x,y)=\sum\limits_{k=0}^{m_{x,y}}\dfrac{1}{k!}\eta^{\ast k}(x,y), \qquad m_{x,y}=\min\{\deg(x),\deg(y)\}, \ x,y\in B.
\]
In particular,
\begin{align}\label{eqn:ealaeta=eta}
e^\eta(x,x_i)=\eta(x,x_i), \qquad x\in B, \ i\in\I_\theta.
\end{align}
\end{remark}

In our setting, we require that $\sigma=e^\eta$ satisfies \eqref{eqn:sigma-intro}.
Next proposition shows that this is to require that $\eta\in \Z^2(B,\k)^H$.

\begin{proposition}\label{pro:invariance}
Let $\eta\in \Z^2(B,\k)$ be a Hochschild 2-cocycle. Then $e^\eta$ is $H$-linear if and only if $\eta\in \Z^2(B,\k)^H$.
\end{proposition}
	\pf
We show that $e^\eta\in Z^2(B)^H\Rightarrow\eta\in \Z^2(B,\k)^H$, the other implication is clear.
We need to show that $\eta(h_{(1)}\cdot x,h_{(2)}\cdot y)=\eps(h)\eta(x,y)$. If $\deg y=0$, then this is automatic.
If $\deg y=1$, say $y=x_i$, then by \eqref{eqn:ealaeta=eta} we have that
\begin{align*}
\eta(h_{(1)}\cdot x,h_{(2)}\cdot x_i)=e^{\eta}(h_{(1)}\cdot x,h_{(2)}\cdot x_i)=\eps(h)e^\eta(x,x_i)=\eps(h)\eta(x,x_i).
\end{align*}
Now, for the general case, assume $\deg y=n>1$ and let us write $y=y'x_i$, $\deg y'=n-1$. Then
\begin{align*}
\eta(h_{(1)}\cdot x, h_{(2)}\cdot y)&=\eta(h_{(1)}\cdot x,h_{(2)}\cdot y' h_{(3)}\cdot x_i)=\eta(h_{(1)}\cdot x h_{(2)}\cdot y', h_{(3)}\cdot x_i)\\
&=
\eta(h_{(1)}\cdot (x y'), h_{(2)}\cdot x_i)=\eps(h)\eta(x y', x_i)=\eps(h)\eta(x, y).
\end{align*}
Hence the result follows.
\epf

\begin{remark}\label{rem:alpha-grado1}
	Set $\tau=e^{\eta}$ in Lemma \ref{lem:cuentaalpha} so $\tau(x,y)=e^{\eta}(x,y)=\eta(x,y)$.
		In particular, \eqref{eqn:alpha-grado1} becomes
\begin{align}\label{eqn:alpha-grado1-eta}
\alpha(x)\alpha(y)-\alpha(xy)=\sigma(x,y)-\eta(x,y).
\end{align}	
	If $x=y$, then this is
	\[\alpha(x)^2-\alpha(x^2)=\sigma(x,x)-\eta(x,x).\]
	Assume that $c(x\ot x)=-x\ot x$, thus $x^2$ is primitive and $\alpha^{-1}(x^2)=-\alpha(x^2)$.
	
	If, moreover, $\alpha$ is $H$-linear, then
	\begin{align*}
	-\alpha(x^2)&=\alpha^{-1}(x^2)\stackrel{\eqref{eqn:alpha-1xy}}{=}\alpha(x)^2-\alpha(x_{(-1)}\cdot x)\alpha^{-1}(x_{(0)})-\alpha(x^2)\\
	&=\alpha(x)^2-\eps(x_{(-1)})\alpha( x)\alpha^{-1}(x_{(0)})-\alpha(x^2)\\
	&=\alpha(x)^2-\alpha( x)\alpha^{-1}(x)-\alpha(x^2)=2\alpha(x)^2-\alpha(x^2).
	\end{align*}
	That is, $\alpha(x)=0$ and 	\eqref{eqn:alpha-grado1-eta} becomes
	\begin{align}\label{eq:condicion-alpha-eta-sigma}
	\alpha(xy)=\eta(x,y)-\sigma(x,y),\qquad x,y\in V.
	\end{align}

	In particular, if $x^2=0$ (v.g.~when $B$ is a Nichols algebra) we obtain
	\begin{align}\label{eq:sigmaxx}
	\sigma(x,x)=\eta(x,x).
	\end{align}
	In other words, $\eta(x,x)=\sigma(x,x)$ is a necessary condition to $e^{\eta}\sim\sigma$.
\end{remark}

\subsection{Equivalence of the commutation conditions}\label{sec:when}
In \cite{GaM}, for a Hopf algebra $A=\Bq(V)\# H$, the authors work with cocycles $\eta$ concentrated in degree $(1,1)$, that is $\eta(A^n,A^m)=0$ if $(n,m)\neq (1,1)$ and show in \cite[Lemma 2.3]{GaM} that \eqref{eqn:conm1} and \eqref{eqn:conm2} are equivalent to each other and in turn to two other identities in $V^{\ot 4}$:
\begin{align}\label{eqn:GaM-V4}
\begin{split}
({\eta}\ot {\eta})(\id\ot c\ot \id)&= ({\eta}\ot {\eta})(\id\ot c\ot \id)(c\ot c),\\
({\eta}\ot {\eta})(\id\ot c\ot \id)(\id\ot \id\ot c)&= ({\eta}\ot {\eta})(\id\ot c\ot \id)(c\ot \id\ot \id).
\end{split}
\end{align}

We are interested in finding more general contexts where this equivalence holds, to reduce the number of identities to check.
We obtain some answers in Lemmas \ref{lem:eta-symmetric} and \ref{lem:conmutation} below.

\begin{lemma}\label{lem:eta-symmetric}
Let $A$ be a Hopf algebra.
If $\eta(x,y)=\lambda\,\eta(y,x)$ for every $x,y\in A$ and some $\lambda\in\k$, then \eqref{eqn:conm1} and  \eqref{eqn:conm2}  are equivalent.
\end{lemma}
\pf
Interchange $a$ and $c$ in one of the lines from \eqref{eqn:conds-in-elements} to obtain the other.
\epf

The following lemma is rather technical, see \eqref{eqn:conmutation} for a concrete and detailed application.
\begin{lemma}\label{lem:conmutation}
	Let $A$ be a Hopf algebra and fix a linear basis $\mathbb{B}$ of $A$.
	Let $\eta\in\Z^2(A,\k)$ be a Hochschild 2-cocycle. Fix $T=\k[\tau_x:x\in\mathbb{B}]$ the polynomial algebra on $\dim A$ variables given by the symbols $\tau_x\coloneqq \eta(-,x)\colon A\to \k$.
	
	In particular, for every $y,z\in A$, there are families of scalars $(u_x=u_x(y,z))_{x\in\mathbb{B}}$, $(d_x=d_x(y,z))_{x\in\mathbb{B}}\in\k$ such that the following identities hold in $T$:
	\begin{align*}
	\eta(-,y_{(1)}z_{(1)})\eta(y_{(2)},z_{(2)})&=\sum_{x\in\mathbb{B}} u_x(y,z)\tau_x, \\ \eta(y_{(1)},z_{(1)})\eta(-,y_{(2)}z_{(2)})&=\sum_{x\in\mathbb{B}} d_x(y,z)\tau_x.
	\end{align*}
	Assume that
	\begin{align}\label{eqn:ua-da}
	u_x(y,z)=d_x(y,z), \qquad \forall x,y,z\in \mathbb{B}.
	\end{align}
	Then \eqref{eqn:conm1} implies \eqref{eqn:conm2}.
	
	The converse is also true, {\it mutatis mutandis}.
\end{lemma}
\pf
Assume that \eqref{eqn:conm1} and \eqref{eqn:ua-da} hold. Let us set $\omega_z\coloneqq \eta(-,z)\colon A\to \k$, for each $z\in\mathbb{B}$.
Now, if $x,y\in A$, then
\begin{align*}
\eta(x_{(1)}y_{(1)}, - )\eta(x_{(2)},y_{(2)})&=\sum_{z\in\mathbb{B}}u_z(x,y)\omega_z\stackrel{\eqref{eqn:ua-da}}{=}\sum_{z\in\mathbb{B}}d_z(x,y)\omega_z\\
&=\eta(x_{(1)},y_{(1)})\eta(x_{(2)}y_{(2)},-).
\end{align*}
Which shows that \eqref{eqn:conm2} also holds.
\epf

\section{An example of abelian type}\label{sec:a2}

In this section we consider the Nichols algebra $\Bq_{\bq}$ associated to a braided vector space $(V,c^{\qb})$ of Cartan type $A_2$, with parameter $q=-1$. The Dynkin diagram is
\[
\Dchaintwo{-1}{-1}{-1}
\]
and the braiding matrix $\qb=(q_{ij})$ is such that $q_{11}=q_{22}=-1$ and $q_{12}q_{21}=-1$. 

We fix a basis $\k\{x_1,x_2\}$ associated to this matrix; then the relations defining the Nichols algebra are given by:
\begin{align}\label{eqn:rels_nichols}
x_1^2&=0, & x_2^2&=0, & x_{12}^2&=0;
\end{align}
where $x_{12}\coloneqq x_1x_2 - q_{12}x_2x_1$.
A linear PBW basis for this algebra is given by
\begin{align}\label{eqn:basis}
\mathbb{B}_{\qb}=\{1,x_1,x_2,x_{12},x_{12}x_1,x_2x_1,x_2x_{12},x_2x_{12}x_1\}.
\end{align}

We assume that there is a principal YD-realization $(\chi_i,g_i)_{i\in\I_2}$ so that $\Bq_{\bq}\in\ydh$ for some Hopf algebra $H$, {\it cf.}~\eqref{eqn:realization}. We set $A=B\# H$ and fix
\begin{align}\label{eqn:Gamma}
\Gamma\coloneqq\lg g_1,g_2\rg\leq G(H).
\end{align}

\subsection{Hopf cocycles}

The (braided) cleft objects for the Nichols algebra $\Bq_{\bq}$ are given by the family of algebras $\mE(\bs\lambda)$, $\bs\lambda=(\lambda_1,\lambda_2,\lambda_{12})\in\k^3$, generated by $y_1,y_2$ and relations
\begin{align}\label{eqn:cleft}
y_1^2&=\lambda_1, & y_2^2&=\lambda_2, & y_{12}^2&=\lambda_{12};
\end{align}
again, $y_{12}\coloneqq y_1y_2 - q_{12}y_2x_1$. The scalars $\lambda_1,\lambda_2,\lambda_{12}\in\k$ are subject to:
\begin{align}\label{eqn:restictions-lambda}
\lambda_i&\neq 0 \text{ only if } \chi_i^2=\eps, i\in\I_2;  & \lambda_{12}&\neq 0 \text{ only if } (\chi_{1}\chi_{2})^2=\eps.
\end{align}
This conditions can also be written as
\begin{align*}
\lambda_jq_{ij}^2&=\lambda_j, \, i,j\in\I_2;  & \lambda_{12}q_{ij}^2&=\lambda_{12}, \ i\neq j\in\I_2.
\end{align*}
This algebra inherits a basis as in \eqref{eqn:basis}.

\begin{remark}\label{rem:q12}
Observe that unless $q_{12}=\pm1$, and thus $q_{21}=\mp 1$, then we have $\lambda_1=\lambda_2=\lambda_{12}=0$.
\end{remark}

For each  triple $\bs\lambda=(\lambda_1,\lambda_2,\lambda_{12})$, the corresponding deformation of $A=\Bq_{\bq}\# H$ is the quotient $\mathfrak{u}_{\qb}(\bs\lambda)$ of $T(V)\# H$ modulo the relations
\begin{align}\label{eqn:rels-hopf}
a_1^2=\lambda_1(1-g_1^2),  a_2^2=\lambda_2(1-g_2^2),  a_{12}^2=\lambda_{12}(1-g_1^2g_2^2)+4q_{12}\lambda_1\lambda_2g_2^2(1-g_1^2).
\end{align}
Here $V=\k\{a_1,a_2\}$ and $a_{12}\coloneqq a_1a_2 - q_{12}a_2a_1$.

\begin{remark}\label{rem:lambda=0}
Notice that when $g_i^2=0$, $i=1,2$, we can normalize $\lambda_i=0$ in the deformations \eqref{eqn:rels-hopf}, on top of the restriction \eqref{eqn:restictions-lambda}. Idem for $g_1^2g_2^2=1$ and $\lambda_{12}$. However, the cleft objects are not isomorphic, cf.~\eqref{eqn:cleft}.
\end{remark}

Henceforward, we will assume that
\begin{equation}\label{eqn:q12}
q_{12}=\pm 1,
\end{equation}
as we shall investigate Hopf 2-cocycles $\sigma_{\bs\lambda}$ associated to non-trivial deformations $\mE(\bs\lambda)$ and $\mathfrak{u}_{\qb}(\bs\lambda)$.

\begin{lemma}\label{lem:gamma}
The section $\gamma\colon\Bq_{\bq}\to\mE(\bs\lambda)$ is given by
\begin{align*}
\gamma(x_i)&=y_i,\  i\in\I_2, &  \gamma(x_{12})&=y_{12}, & \gamma(x_{2}x_1)&=y_{2}y_1,\\
\gamma(x_{12}x_1)&=y_{12}y_{1} - 2q_{21}\lambda_1 y_2, & \gamma(x_2x_{12})&=y_{2}y_{12},& \gamma(x_2x_{12}x_1)&=y_2y_{12}y_1.
\end{align*}
The convolution inverse for $\gamma$ is determined by
\begin{align*}
\gamma^{-1}(x_i)&=-y_i,\  i\in\I_2, &  \gamma^{-1}(x_{12})&=y_{12}+2q_{12}y_2y_1, \\
\gamma^{-1}(x_{2}x_1)&=q_{21}y_{12}-y_2y_1, &
\gamma^{-1}(x_{12}x_1)&=-y_{12}y_1+2q_{21}\lambda_1y_2, \\
 \gamma^{-1}(x_2x_{12})&=-y_2y_{12},& \gamma^{-1}(x_2x_{12}x_1)&=y_2y_{12}y_1-\lambda_{12}.
\end{align*}
\end{lemma}
\pf
The values for $\gamma(x_i), \gamma(x_{12}), \gamma(x_{2}x_1)$ are straightforward. We see that
\begin{align*}
\underline{\Delta}(x_{12}x_1)&=-q_{21}x_1\ot x_{12}+x_{12}\ot x_1+2x_1\ot x_2x_1,\\
\underline{\rho}(y_{12}y_1)&=2q_{21}\lambda_1\,1\ot x_2 -q_{21}y_1\ot x_{12}+x_{12}\ot x_1+2y_1\ot x_2x_1
\end{align*}
and thus setting $\gamma(x_{12}x_1)=(y_{12}y_{1} - 2q_{21}\lambda_1 y_2 )$ gives $(\gamma\otimes \id)\Delta(x_{12}x_1)=\rho\gamma(x_{12}x_1)$ as expected.
On the other hand, it is easy to see $\gamma(x_2x_{12})=y_{2}y_{12}$.
Finally,
\begin{align*}
\underline{\Delta}(x_2x_{12}x_1)&=2x_2x_1\ot x_2x_1 - q_{21}^{2}x_1\ot x_2x_{12} - q_{21}x_{12}\ot x_2x_1 \\
&- q_{21}^{2}x_{12}x_1\ot x_2
+ x_2\ot x_{12}x_1 - q_{21}x_2x_1\ot x_{12} + x_2x_{12}\ot x_1,\\
\underline{\rho}(y_2y_{12}y_1)&=2q_{21}\lambda_1 y_2\ot x_2+
2y_2y_1\ot x_2x_1 - q_{21}^{2}y_1\ot x_2x_{12} - q_{21}y_{12}\ot x_2x_1 \\
&- q_{21}^{2}y_{12}y_1\ot x_2
+ y_2\ot x_{12}x_1 - q_{21}y_2y_1\ot x_{12} + y_2y_{12}\ot x_1.
\end{align*}
So $(\gamma\ot\id)\Delta(x_2x_{12}x_1)=\rho(y_2y_{12}y_1)$ as (recall $\lambda_1q_{21}^3=\lambda_1 q_{21}$):
\begin{align*}
(\gamma\ot\id)&\underline{\Delta}(x_2x_{12}x_1)
=2y_2y_1\ot x_2x_1 - q_{21}^{2}y_1\ot x_2x_{12} - q_{21}y_{12}\ot x_2x_1 \\
&- q_{21}^{2}(y_{12}y_{1} - 2q_{21}\lambda_1 y_2 )\ot x_2
+ y_2\ot x_{12}x_1 - q_{21}y_2y_1\ot x_{12} + y_2y_{12}\ot x_1.
\end{align*}
Now, to see that $\gamma^{-1}$ is the convolution inverse of $\gamma$, we only need to verify that $\gamma\ast\gamma^{-1}(b)=0=\gamma^{-1}\ast\gamma(b)$, for each $b\in\mathbb{B}_{\bq}$. To illustrate this, we expose some calculations below.
\begin{align*}
\gamma\ast\gamma^{-1}(x_{12}x_1)=&\gamma(x_{12}x_1)-q_{21}y_1(y_{12}+2q_{12}y_2y_1)-y_{12}y_1+2y_1(q_{21}y_{12}-y_2y_1)\\
&+\gamma^{-1}(x_{12}x_1)\\
=&-y_{12}y_1+2y_1y_2y_1-y_{12}y_1+2y_{12}y_1-2y_1y_2y_1=0.\\
\gamma\ast\gamma^{-1}(x_{2}x_{12}x_1) =& y_2y_{12}y_1+2y_2y_1(q_{21}y_{12}-y_2y_1)+q_{21}^2y_1y_2y_{12}-y_2y_{12}y_1
\\
&-q_{21}y_{12}(q_{21}y_{12}-y_2y_1)+q_{21}^2(y_{12}y_1-2q_{21}\lambda_1y_2)y_2+y_2y_{12}y_1\\
&-q_{21}^2\lambda_{12}+y_2(-y_{12}y_1+2q_{21}\lambda_1y_2)-q_{21}y_2y_1(y_{12}+2q_{12}y_2y_1)\\
=&2q_{21}y_2y_1y_{12}+q_{21}^2(y_{12}+q_{12}y_2y_1)y_{12}
-q_{21}^2\lambda_{12}+q_{21}y_{12}y_2y_1\\
&+q_{21}^2y_{12}y_1y_2-2q_{21}\lambda_1\lambda_2-q_{21}^2\lambda_{12}+2q_{21}\lambda_1\lambda_2-q_{21}y_2y_1y_{12}\\
=&2y_2y_{12}y_1-q_{21}y_2y_1y_{12}
+q_{21}y_{12}y_2y_1+q_{21}^2y_{12}(y_{12}+q_{12}y_2y_1)\\
&-q_{21}^2\lambda_{12}-y_2y_{12}y_1\\
=&y_2y_{12}y_1-y_2y_{12}y_1=0.
\end{align*}
The rest of the computations, which are simpler, are left to the reader.
\epf

\begin{remark}
Notice that $\gamma\colon\Bq_{\bq}\to\mE(\bs\lambda)$ as in Lemma \ref{lem:gamma} is $H$-linear. Indeed, it suffices to check this on $\gamma(x_{12}x_1)$ and this follows since
\[
\lambda_1 h\cdot y_2=\lambda_1\chi_2(h)y_2= \lambda_1\chi_1^2(h)\chi_2(h)y_2, \quad h\in H.
\]
\end{remark}

\begin{theorem}\label{thm:sigma-a2}Let $\bs\lambda=(\lambda_1,\lambda_2,\lambda_{12})$.
The Hopf cocycle $\sigma_{\bs\lambda}\coloneqq\sigma$ associated to the cleft object $\mE(\bs\lambda)$ is given by
\begin{center}
	\begin{table}[H]
		\resizebox{12cm}{!}
		 {\begin{tabular}{|c|c|c|c|c|c|c|c|c|}
				\hline
				$\sigma$ &  $x_1$ & $x_2$ & $x_{12}$ & $x_2x_1$  &  $x_2x_{12}$ & $x_{12}x_1$ & $x_2x_{12}x_1$ \\
				\hline
				$x_1$   & $\lambda_1$ & $0$ & $0$ & $0$ & $\lambda_{12}$ & $0$ & $0$ \\
				\hline
				$x_2$  & $0$ & $\lambda_2$ & $0$ & $0$  &  $0$ & $2q_{12}\lambda_1\lambda_2$& $0$\\
				\hline
				$x_{12}$  & $0$ & $0$ & $\lambda_{12}$ & $0$ &  $0$ & $0$ & $0$ \\
				\hline
				$x_2x_1$  & $0$ & $0$ & $0$ & $-q_{21}\lambda_1\lambda_2$  &  $0$ & $0$ & $0$   \\
				\hline
				$x_2x_{12}$  & $0$ & $0$ & $0$ & $0$& $-q_{12}\lambda_2\lambda_{12}$ & $0$ & $0$  \\
				\hline
				$x_{12}x_1$  & $0$ & $2q_{12}\lambda_1\lambda_2 + \lambda_{12}$ & $0$ & $0$ &  $0$ & $q_{12}\lambda_{12}\lambda_1+4\lambda_1^2\lambda_2$ & $0$  \\
				\hline
				$x_2x_{12}x_1$ & $0$ & $0$ & $0$ & $0$  &  $0$ & $0$ & $q_{12}\lambda_2\lambda_{12}\lambda_1$ \\
				\hline
		\end{tabular}}
	\end{table}
\end{center}
\end{theorem}
\pf
In first place, we use Lemma \ref{lem:cuentacociclo} to calculate the elements of $S_\sigma$ as in \eqref{eqn:S-sigma}. Indeed, let $i,j\in\I_2$, then
\begin{align*}
  \sigma(x_i,x_j)=-q_{ij}y_jy_i+\gamma^{-1}(x_ix_j)=\delta_{ij}\lambda_i,
\end{align*}
here we use that $\gamma^{-1}(x_ix_j)=q_{ij}y_jy_i$.
It is easy to see that
\[
\sigma(x_i,x_{12})=\sigma(x_i,x_2x_1)=0.
\]
For the other hand, we show $\sigma(x_1,x_2x_{12})$ and $\sigma(x_2,x_{12}x_1)$.
\begin{align*}
\sigma(x_1,&x_2x_{12})=y_2y_{12}y_2-2q_{12}y_2y_1\gamma^{-1}(x_1x_2)-y_2\gamma^{-1}(x_{12}x_1)-y_{12}\gamma^{-1}(x_1x_2)\\
&\qquad \qquad +\gamma^{-1}(x_1x_2x_{12})\\ =&y_2y_{12}y_2-2y_2y_1y_2y_1+y_2(y_{12}-2q_{21}\lambda_1y_2)-q_{12}y_{12}y_2y_1-y_1y_{12}y_2+\lambda_{12}\\
=&y_2y_{12}y_1-2y_2y_{12}y_1-2q_{12}\lambda_1\lambda_2+y_2y_{12}y_1+2q_{12}\lambda_1\lambda_2+\lambda_{12}=\lambda_{12}.
\end{align*}
\begin{align*}
\sigma(x_2,&x_{12}x_1)=q_{21}^2(y_{12}y_1-2q_{21}\lambda_1y_2)y_2+y_1y_2y_{12}\\
&\qquad \qquad -q_{21}y_{12}\gamma^{-1}(x_2x_1)+\gamma^{-1}(x_2x_{12}x_1)\\
=&y_{12}y_1y_2-2q_{21}\lambda_1\lambda_2+\lambda_{12}-y_2y_{12}y_1-y_{12}y_1y_2+y_2y_{12}y_1+\lambda_{12}=2q_{12}\lambda_1\lambda_2.
\end{align*}
The reader will have no difficulty to see that
\[
\sigma(x_1,x_{12}x_1)=\sigma(x_2,x_2x_{12})=\sigma(x_i,x_1x_{12}x_2)=0.
\]
Thus, $S_\sigma=\{0,\lambda_1,\lambda_2,\lambda_{12},2q_{12}\lambda_1\lambda_2\}$. From now on, we use the formulas of decomposition in Example \ref{exa:descompos-lowdregree} to get the others values of $\sigma$, following Lemma \ref{lem:descomposi}. In this case, the matrix coefficients of the braiding $(c_{i,j}^{p,q})_{
	\begin{smallmatrix}
	i,j,\\ p,q
	\end{smallmatrix}
	\in\I}$ are given by $c_{i,j}^{p,q}=\delta_{i,q}\delta_{j,p}q_{ij}$ so $\mathbb{Z}[c_{i,j}^{p,q}]_{\begin{smallmatrix}
		i,j,\\ p,q
	\end{smallmatrix}
	\in\I}=\mathbb{Z}[q_{12}^{\pm 1}]\simeq \mathbb{Z}$, as $q_{11}=q_{22}=-1$ and $q_{21}=-q_{12}^{-1}=\mp1$.

One example of application is showed below.
\begin{align*}
 \sigma(x_{12}x_1,&x_{12}x_1)=\sigma(x_1,x_2x_1x_{12}x_1)+\sigma(x_2,g_1\cdot (x_{12}x_1))\sigma(x_1,x_1)\\
 &\quad+\sigma(g_2\cdot x_1,g_2\cdot (x_{12}x_1))\sigma(x_1,x_2)-\sigma(x_1,x_2)\sigma(x_1,x_{12}x_1)\\
 &\quad-\sigma(x_1,g_2\cdot x_1)\sigma(x_2,x_{12}x_1)+\sigma(x_2x_1,2x_1)\sigma(x_1,x_2x_1)\\
 &\quad-\sigma(x_2x_1,q_{21}x_1)\sigma(x_1,x_{12})+\sigma(x_2x_1,x_{12})\sigma(x_1,x_1)\\
 &\quad+\sigma(x_2,g_1\cdot(2x_1))\sigma(x_1,x_1x_2x_1)-\sigma(x_2,g_1\cdot(q_{21}x_1))\sigma(x_1,x_1x_{12}x_1)\\
 &\quad+\sigma(x_2,g_1\cdot x_{12})\sigma(x_1,x_1^2)+\sigma(g_2\cdot x_1,g_2\cdot (2x_1))\sigma(x_1,x_2^2x_1)\\
 &\quad-\sigma(g_2\cdot x_1,g_2\cdot (q_{21}x_1))\sigma(x_1,x_2x_{12})+\sigma(g_2\cdot x_1,g_2\cdot x_{12})\sigma(x_1,x_2x_1)\\
 &=q_{12}\sigma(x_2, x_{12}x_1)\sigma(x_1,x_1)-q_{21}\sigma(x_1, x_1)\sigma(x_2,x_{12}x_1)\\
 &\quad-q_{21}\sigma( x_1,x_1)\sigma(x_1,x_2x_{12})\\
 &=2q_{12}\sigma(x_2, x_{12}x_1)\sigma(x_1,x_1)+q_{12}\sigma( x_1,x_1)\sigma(x_1,x_2x_{12})\\
&=4\lambda_1^2\lambda_2+q_{12}\lambda_{12}\lambda_1.
\end{align*}
The other values are computed analogously.
\epf

\subsection{Hochschild cocycles}\label{sec:hoch-a2}

Recall the cocycles $\xi_j^i\in \Z^2(\Bq_{\bq},\k)$, $i,j\in\I_2$, from Example \ref{exa:xi_ij}.
Additionally, we denote  $\xi_0=\eps\ot\eps\in \Z^2(\Bq_{\bq},\k)$, the cocycle such that $\xi_0(1,1)=1$, and zero elsewhere.

We shall consider two more Hochschild 2-cocycles $\xi_{121}^2,\xi_{212}^1$ in $\Bq_{\bq}$. Namely,
let $\xi_{121}^2\colon \Bq_{\bq}\ot \Bq_{\bq}\to \k$ be the linear map given by
\begin{align}\label{eqn:xi-121}
\begin{split}
\xi_{121}^2(x_2,x_{12}x_1)&=\xi_{121}^2(x_2x_{12},x_1)=\xi_{121}^2(x_2x_1,x_2x_1)=\xi_{121}^2(x_{12},x_{12})=1,\\
\xi_{121}^2(x_{12},x_2x_1)&=\xi_{121}^2(x_2x_1,x_{12})=-q_{12}.
\end{split}
\end{align}
and zero elsewhere in the basis $\mathbb{B}\ot \mathbb{B}$. 
Following Lemma  \ref{lem:hoch}, it is an easy computation to check that $\xi_{121}^2\in \Z^2(\Bq_{\bq},\k)$. 
As well, we define 
$\xi_{212}^1\in \Z^2(\Bq_{\bq},\k)$ via
\begin{align}\label{eqn:xi-212}
\xi_{212}^1(x_1,x_2x_{12})=\xi_{212}^1(x_{12},x_{12})=\xi_{212}^1(x_{12}x_{1},x_2)=1
\end{align}
and zero elsewhere in $\mathbb{B}\ot \mathbb{B}$.

We further note that $\xi_{121}^2+\xi_{212}^1\in\B^2(B,\k)$. Indeed, let $f\colon B\to \k$ be the linear map defined on the basis $\mathbb{B}$ by $f(x_2x_{12}x_1)=1$ and zero elsewhere. Then 
	\begin{align}\label{eqn:cobordism}
	(\xi_{121}^2-\xi_{212}^1)(b,b')=f(bb'), \qquad b,b'\in\mathbb{B}.
	\end{align}

Recall the group $\Gamma\leq G(H)$ from \eqref{eqn:Gamma}.
\begin{lemma}\label{lem:generadores-eta-abeliano}
	$\Z^2(\Bq_{\bq},\k)^\Gamma$ is spanned  over $\k$ by $\xi_0, \xi_1^1,\xi_2^2,\xi_{121}^2, \xi_{212}^1$.
	
Furthermore,
\begin{enumerate}[label=(\alph*)]
\item $\xi_i^i\in \Z^2(\Bq_{\bq},\k)^H$ if and only if $\chi_i^2=\eps$,
\item $\xi_{iji}^j\in \Z^2(\Bq_{\bq},\k)^H$ if and only if $(\chi_{1}\chi_{2})^2=\eps$; $i,j\in\I_2$.
\end{enumerate}
\end{lemma}
\pf
%
By Lemma \ref{lem:hoch}, see Remark \ref{rem:hoch}, $\eta\in \Z^2(\Bq_{\bq},\k)$ is determined by the values
\begin{align*}
\eta&(x_1,x_1), & \eta&(x_1,x_2), & \eta&(x_{1},x_{12}), & \eta&(x_1,x_2x_{12}), \\
\eta&(x_1,x_2x_1), & \eta&(x_2,x_2), & \eta&(x_{2},x_1), & \eta&(x_2,x_{12}), \\
\eta&(x_2,x_{12}x_1) & \eta&(1,1).
\end{align*}
Furthermore, as $x_1x_2x_1=x_{12}x_1$ and $x_2x_1x_2=x_2x_{12}$, we get
\begin{align*}
\eta(x_1,x_{12}x_1)=\eta(x_2,x_2x_{12})=0.
\end{align*}
Also, it is easy to see that
\begin{align*}
\eta(x_1,x_2x_{12}x_1)&=\eta(x_2,x_2x_{12}x_1)=\eta(x_2,x_2x_{12})=\eta(x_2,x_2x_1)=0,\\
\eta(x_1,x_2x_1)&=q_{21}\eta(x_1,x_{12}).
\end{align*}
Now we let $\Gamma$ act and obtain
\[
\eta(x_1,x_2)=\eta(g_1\cdot x_1,g_1\cdot x_2)=-q_{12}\eta(x_1,x_2)
\]
which forces $q_{12}=-1$ but also
\[
\eta(x_1,x_2)=\eta(g_2\cdot x_1,g_2\cdot x_2)=-q_{21}\eta(x_1,x_2),
\]
which also gives $q_{21}=-1$, a contradiction to $q_{12}q_{21}=-1$. Thus $\eta(x_1,x_2)=0$, similarly $\eta(x_2,x_1)=0$.
Similarly, we need to have
\[
-q_{21}^2\eta(x_{1},x_{12})=\eta(x_{1},x_{12}),
\]
which gives $\eta(x_{1},x_{12})=0$; and similarly
\[
-q_{12}^2\eta(x_{2},x_{12})=\eta(x_{2},x_{12}), \ -q_{12}^2\eta(x_2x_1,x_2)=\eta(x_2x_1,x_2)
\]
so $\eta(x_{12},x_2)=0$.
That is, we are left with the values
\begin{align*}
\eta&(1,1), &\eta&(x_1,x_1), & \eta&(x_2,x_2), & \eta&(x_2,x_{12}x_1), & \eta&(x_1,x_2x_{12})
\end{align*}
and therefore $\eta\in\k\{\xi_0, \xi_1^1,\xi_2^2,\xi_{121}^2, \xi_{212}^1\}$.

The second statement is straightforward; recall  $h\cdot x_i=\chi_i(h)x_i$, $h\in H, i\in\I_\theta$.\epf

\subsection{Exponentials}

Let $\eta\in \Z^2(\Bq_{\bq},\k)^H$, we are interested in finding the conditions for $\eta$ such that $e^\eta$ is a Hopf cocycle; set
\begin{align}\label{eq:eta-abeliano}
\eta=\eta_1\xi_1^1+ \eta_2\xi^2_2+\eta_{121}\xi_{121}^2+\eta_{212}\xi_{212}^1.
\end{align}
In particular, $\eta(x_i,x_i)=\eta_i$, $i\in\I_2$, $\eta(x_{2},x_{12}x_1)=\eta_{121}$ and $\eta(x_{1},x_2x_{12})=\eta_{212}$.

Next proposition characterizes the cocycles $\eta$ that satisfy \eqref{eqn:conm1} and \eqref{eqn:conm2}; so in particular $e^{\eta}$ is a Hopf cocycle.
It is enough to assume that $\eta\in \Z^2(\Bq_{\bq},\k)^\Gamma$.

\begin{proposition}\label{pro:exponencial-abeliano}
A cocycle $\eta\in\Z^2(\Bq_{\bq},\k)^\Gamma$ satisfies \eqref{eqn:conm1} and \eqref{eqn:conm2} if and only if
\[
\eta \in \C\coloneqq\{\eta_1\xi_1^1,\eta_2\xi_2^2, \eta_{121}\xi_{121}^2+\eta_{212}\xi_{212}^1:\eta_1,\eta_2,\eta_{121},\eta_{212}\in\k\}.
\]
In particular if $\eta\in \C$, then $e^\eta$ is a Hopf 2-cocycle.
\end{proposition}
We will show in Lemma \ref{lem:converse}  that conditions \eqref{eqn:conm1} and \eqref{eqn:conm2} are not necessary for $e^\eta$ to be a Hopf 2-cocycle.

\begin{remark}
	Let $\eta=\eta_1\xi_1^1+ \eta_2\xi^2_2+\eta_{121}\xi_{121}^2+\eta_{212}\xi_{212}^1\in \Z^2(\Bq_{\bq},\k)^\Gamma$.
	We can rephrase Proposition \ref{pro:exponencial-abeliano} as follows:
 $\eta\in \C$ if and only if
  \begin{align}\label{eq:eta-Ceta}
    \eta_i\eta_j=\eta_i\eta_{iji}=\eta_i\eta_{jij}=0,\qquad i\neq j\in \I_2.
  \end{align}
\end{remark}

\pf
Suppose that $\eta$ verifies \eqref{eqn:conm1} and \eqref{eqn:conm2}.
That is for any $x,y,z\in \Bq_{\bq}$,
\begin{align}\label{eq:condicionsuficienteeta}
\eta(x,y_{(1)}z_{(1)})\eta(y_{(2)},z_{(2)})&=\eta(x,y_{(2)}z_{(2)})\eta(y_{(1)},z_{(1)}),\\
\label{eq:condicionsuficienteeta2}
\eta(x_{(1)}y_{(1)},z)\eta(x_{(2)},y_{(2)})&=\eta(x_{(2)}y_{(2)},z)\eta(x_{(1)},y_{(1)}).
\end{align}
Now, we see that if $y=x_2x_1$ and $z=x_2$, then \eqref{eq:condicionsuficienteeta} gives
\begin{align}\label{eq:condicioneta11}
q_{21}\eta(x,x_1)\eta(x_2,x_2)=q_{12}\eta(x,x_1)\eta(x_2,x_2), \quad\text{for any $x\in\Bq_{\bq}$.}
\end{align}
Similarly, if $y=x_1x_2$ and $z=x_1$, then \eqref{eq:condicionsuficienteeta} gives
\begin{align}\label{eq:condicioneta22}
q_{21}\eta(x,x_2)\eta(x_1,x_1)=q_{12}\eta(x,x_2)\eta(x_1,x_1),\quad\text{for any $x\in \Bq_{\bq}$.}
\end{align}
Thus, from \eqref{eq:condicioneta11} and \eqref{eq:condicioneta22}, we obtain
\begin{align*}
    \eta(x_1,x_1)\eta(x_2,x_2)=\eta(x_2,x_{12}x_1)\eta(x_2,x_2)=\eta(x_1,x_2x_{12})\eta(x_1,x_1)=0.
\end{align*}
If we suppose that $\eta(x_1,x_1)\neq 0$, then $\eta(x_2,x_2)=\eta(x_1,x_2x_{12})=0$.
Furthermore, if $y=x_{12}x_1$ and $z=x_1$, then we obtain from \eqref{eq:condicionsuficienteeta} that
\begin{align*}
\eta(x,x_{12})\eta(x_1,x_1)=-2q_{21}\eta(x_1,x_1)\eta(x,x_2x_1)+\eta(x_1,x_1)\eta(x,x_{12}),
\end{align*}
for any $x\in\Bq_{\bq}$.
Hence, we get that $\eta(x_2,x_{12}x_1)=0$.
Through the same calculations for $\eta(x_2,x_2)\neq 0$, we get that $\eta(x_1,x_1)=\eta(x_2,x_{12}x_1)=\eta(x_1,x_2x_{12})=0$. Thus, $\eta\in \C$.

Conversely, suppose that $\eta\in \C$. On one hand, if $\eta=\eta_1\xi_i^i$, $i\in\I_2$, then by \cite[Lemma 2.3.]{GaM}, we have the conclusion as $\eta$ clearly satisfies \eqref{eqn:GaM-V4}.

 Now, we show that the hypothesis of Lemma \ref{lem:conmutation} hold when
 \begin{align}\label{eqn:conmutation}
 \eta=\eta_{121}\xi_{121}^2+\eta_{212}\xi_{212}^1.
 \end{align}
 Let $x,y,z\in\mathbb{B}_\qb$, such that $\deg(y), \deg(z)$ are positive. In first place, note that if $\deg(y)+\deg(z)<4$, then we obtain zero in both sides of \eqref{eqn:conm1} and thus $u_x(y,z)=d_x(y,z)=0$. Now, if $\deg(y)+\deg(z)=4$, then \eqref{eqn:conm1} becomes $\eta(x,1)\eta(y,z)$ and so, $u_x(y,z)=0=d_x(y,z)$ when $x\neq 1$ and  $u_1(y,z)=\eta(y,z)=d_1(y,z)$. On the other hand, we explicitly show the most representative calculations below.

\begin{description}[leftmargin=*]
\item[$\deg(y)+\deg(z)=5$]

\

\begin{enumerate}[label=(\roman*),leftmargin=*]
\item $y=x_{12}$ and $z=x_{12}x_1$:
\begin{align*}
\eta(-,y_1z_1)\eta(y_2,z_2)=&2\eta(-,x_1)\eta(x_2,x_{12}x_1)-2q_{21}\eta(-,x_1)\eta(x_{12},x_2x_1)\\
&+\eta(-,x_1)\eta(x_{12},x_{12}),\\
\eta(y_1,z_1)\eta(-,y_2z_2)=&\eta(x_{12},x_{12})\eta(-,x_1).
\end{align*}
Here
\begin{itemize}
    \item $u_{x_1}(x_{12},x_{12}x_1)=\eta(x_{12},x_{12})=d_{x_1}(x_{12},x_{12}x_1)$.
    \item $u_x(x_{12},x_{12}x_1)=d_x(x_{12},x_{12}x_1)=0$ for $x\neq x_1$.
\end{itemize}
\item $y=x_{12}$ and $z=x_2x_{12}$:
\begin{align*}
\eta(-,y_1z_1)\eta(y_2,z_2)=&\eta(-,x_2)\eta(x_1,x_2x_{12})-q_{12}\eta(-,x_2)\eta(x_2x_1,x_{12}),\\
\eta(y_1,z_1)\eta(-,y_2z_2)=&-q_{21}\eta(x_2x_1,x_{12})\eta(-,x_2)+\eta(x_1,x_2x_{12})\eta(-,x_2)\\
&+2\eta(x_2x_1,x_2x_1)\eta(-,x_2).
\end{align*}
Then
\begin{itemize}
    \item $u_{x_2}(x_{12},x_2x_{12})=\eta(x_{12},x_{12})=d_{x_2}(x_{12},x_2x_{12})$.
    \item $u_x(x_{12},x_2x_{12})=d_x(x_{12},x_2x_{12})=0$ for $x\neq x_2$.
\end{itemize}
\item $y=x_2x_1$ and $z=x_{12}x_{1}$:
\begin{align*}
\eta(-,y_1z_1)\eta(y_2,z_2)=&q_{21}\eta(-,x_1)\eta(x_2,x_{12}x_1)-2q_{21}\eta(-,x_1)\eta(x_2x_1,x_2x_1)\\
&+\eta(-,x_1)\eta(x_2x_1,x_{12}),
\\
\eta(y_1,z_1)\eta(-,y_2z_2)=&\eta(x_2x_1,x_{12})\eta(-,x_1)+q_{12}\eta(x_2,x_{12}x_1)\eta(-,x_1).
\end{align*}
In this case $u_{x}(x_{12},x_2x_{12})=0=d_{x}(x_{12},x_2x_{12})$ for any $x\in \mathbb{B}_\qb$.
\item $y=x_{2}x_1$ and $z=x_{2}x_{12}$:
\begin{align*}
\eta(-,y_1z_1)\eta(y_2,z_2)=&\eta(-,x_2)\eta(x_1,x_2x_{12})-q_{12}\eta(-,x_2)\eta(x_2x_1,x_{12}),
\\
\eta(y_1,z_1)\eta(-,y_2z_2)=&-q_{21}\eta(x_2x_1,x_{12})\eta(-,x_2)+\eta(x_1,x_2x_{12})\eta(-,x_2)\\
&+2\eta(x_2x_1,x_2x_1)\eta(-,x_2).
\end{align*}
Then,
\begin{itemize}
    \item $u_{x_2}(x_{2}x_1,x_2x_{12})=\eta(x_{12},x_{12})=d_{x_2}(x_{2}x_1,x_2x_{12})$.
    \item $u_x(x_{2}x_1,x_2x_{12})=d_x(x_{2}x_1,x_2x_{12})=0$ for $x\neq x_2$.
\end{itemize}
\item $y=x_{2}$ and $z=x_{2}x_{12}x_1$:
\begin{align*}
\eta(x,y_1z_1)\eta(y_2,z_2)=&-\eta(x,x_2)\eta(x_2,x_{12}x_1),
\\
\eta(y_1,z_1)\eta(x,y_2z_2)=&-\eta(x_2,x_{12}x_1)\eta(x,x_2).
\end{align*}
Then,
\begin{itemize}
    \item $u_{x_2}(x_{2},x_2x_{12}x_1)=-\eta(x_{2},x_{12}x_1)=d_{x_2}(x_{2},x_2x_{12}x_1)$.
    \item $u_x(x_{2},x_2x_{12}x_1)=d_x(x_{2},x_2x_{12}x_1)=0$ for $x\neq x_2$.
\end{itemize}
\end{enumerate}

\end{description}

\begin{description}[leftmargin=*]
  \item[$\deg(y)+\deg(z)=6$]

  \

  \begin{enumerate}[label=(\roman*),leftmargin=*]
    \item $y=x_{12}$ and $z=x_2x_{12}x_1$:
\begin{align*}
\eta(-,y_1z_1)\eta(y_2,z_2)=&q_{21}\eta(-,x_2x_{1})\eta(x_{12},x_{12}),
\\
\eta(y_1,z_1)\eta(-,y_2z_2)=&\eta(x_{12},x_2x_1)\eta(-,x_2x_1)+q_{21}\eta(x_1,x_2x_{12})\eta(-,x_2x_1).
\end{align*}
Then,
\begin{itemize}
    \item $u_{x_2x_1}(x_{12},x_2x_{12}x_1)=q_{21}\eta(x_{12},x_{12})=d_{x_2x_1}(x_{12},x_2x_{12}x_1)$.
    \item $u_x(x_{12},x_2x_{12}x_1)=d_x(x_{12},x_2x_{12}x_1)=0$ for $x\neq x_2x_1$.
\end{itemize}
    \item $y=x_{2}x_1$ and $z=x_2x_{12}x_1$:
\begin{align*}
\eta(-,y_1z_1)\eta(y_2,z_2)=&\eta(-,x_2x_1)\eta(x_1,x_2x_{12})-q_{21}\eta(-,x_{12})\eta(x_2,x_{12}x_1)\\
&+q_{21}\eta(-,x_{12})\eta(x_2x_1,x_2x_1),
\\
\eta(y_1,z_1)\eta(-,y_2z_2)=&\eta(x_1,x_2x_{12})\eta(-,x_2x_1).
\end{align*}
Thus,
\begin{itemize}
    \item $u_{x_2x_1}(x_{2}x_1,x_2x_{12}x_1)=\eta(x_{1},x_2x_{12})=d_{x_2x_1}(x_{2}x_1,x_2x_{12}x_1)$.
    \item $u_x(x_{2}x_1,x_2x_{12}x_1)=d_x(x_{2}x_1,x_2x_{12}x_1)=0$ for $x\neq x_2x_1$.
\end{itemize}
    \item $y=x_{12}x_1$ and $z=x_2x_{12}$:
\begin{align*}
\eta(-,y_1z_1)\eta(y_2,z_2)=&-2q_{12}\eta(-,x_1x_2)\eta(x_2x_1,x_{12})-\eta(-,x_1x_2)\eta(x_{12},x_{12})\\
&+\eta(-,x_{12})\eta(x_1,x_2x_{12}),
\\
\eta(y_1,z_1)\eta(-,y_2z_2)=&-2q_{12}\eta(x_{12},x_2x_1)\eta(-,x_1x_2)-\eta(x_{12},x_{12})\eta(-,x_{1}x_2)\\
&+\eta(x_1,x_2x_{12})\eta(-,x_{12}).
\end{align*}
Then,
\begin{itemize}
    \item $u_{x_{12}}(x_{12}x_1,x_2x_{12})=-q_{12}\eta(x_2x_1,x_{12})=d_{x_{12}}(x_{12}x_1,x_2x_{12})$.
    \item $u_{x_{2}x_1}(x_{12}x_1,x_2x_{12})=-\eta(x_2x_1,x_{12})=d_{x_{2}x_1}(x_{12}x_1,x_2x_{12})$.
    \item$u_x(x_{12}x_1,x_2x_{12})=d_x(x_{12}x_1,x_2x_{12})=0$ for $x\neq x_{12},x_2x_1$.
\end{itemize}
    \item $y=x_{12}x_1=z$ or $y=x_2x_{12}=z$. In this cases, we obtain zero in both sides of \eqref{eqn:conm1} and so $u_x(y,z)=0=d_x(y,z)$.
  \end{enumerate}
\end{description}

\begin{description}[leftmargin=*]
  \item[$\deg(y)+\deg(z)=7$]

  \

  \begin{enumerate}[label=(\roman*),leftmargin=*]
    \item  $y=x_{12}x_1$ and $z=x_2x_{12}x_1$:
\begin{align*}
\eta(-,y_1z_1)\eta(y_2,z_2)=&-4\eta(-,x_1x_2x_1)\eta(x_2x_1,x_2x_1)+2q_{21}\eta(-,x_1x_{12})\eta(x_2x_1,x_2x_1)\\
&+2q_{21}\eta(-,x_1x_2x_1)\eta(x_2x_1,x_{12})-\eta(-,x_1x_{12})\eta(x_{12},x_2x_1)\\
&+2q_{21}\eta(-,x_1x_2x_1)\eta(x_{12},x_2x_1)-\eta(-,x_1x_2x_1)\eta(x_{12},x_{12})\\
&+\eta(-,x_{12}x_1)\eta(x_1,x_2x_{12})+\eta(-,x_{12}x_1)\eta(x_{12}x_1,x_2),
\\
\eta(y_1,z_1)\eta(-,y_2z_2)=&\eta(x_{12}x_1,x_2)\eta(-,x_{12}x_1)+\eta(x_1,x_2x_{12})\eta(-,x_{12}x_1)\\
&-2q_{12}\eta(x_{12},x_2x_1)\eta(-,x_1x_2x_1)-\eta(x_{12},x_{12})\eta(-,x_1x_2x_1)\\
&-\eta(x_{12},x_2x_1)\eta(-,x_1x_{12}).
\end{align*}
Here we have
\begin{itemize}
    \item $u_{x_{12}x_1}(x_{12}x_1,x_2x_{12}x_1)=2\eta(x_1,x_2x_{12})=d_{x_{12}x_1}(x_{12}x_1,x_2x_{12}x_1)$.
    \item $u_x(x_{12}x_1,x_2x_{12}x_1)=d_x(x_{12}x_1,x_2x_{12}x_1)=0$ for $x\neq x_{12}x_1$.
\end{itemize}
    \item  $y=x_2x_{12}$ and $z=x_2x_{12}x_1$:
\begin{align*}
\eta(-,y_1z_1)\eta(y_2,z_2)=&q_{21}\eta(-,x_{12}x_2)\eta(x_2,x_{12}x_1)-2\eta(-,x_2x_1x_2)\eta(x_2,x_{12}x_1)\\
&+q_{21}\eta(-,x_2x_{12})\eta(x_{12},x_2x_1)-\eta(-,x_2x_{12})\eta(x_2x_{12},x_1),
\\
\eta(y_1,z_1)\eta(-,y_2z_2)=&-\eta(x_2x_{12},x_1)\eta(-,x_2x_{12})-q_{21}\eta(x_{12},x_2x_1)\eta(-,x_2x_{12})\\
&+2\eta(x_2x_1,x_2x_1)\eta(-,x_2x_{12})-q_{21}\eta(x_{12},x_{2}x_1)\eta(-,x_2x_{12}).
\end{align*}
Then,
\begin{itemize}
    \item $u_{x_2x_{12}}(x_{2}x_{12},x_2x_{12}x_1)=-\eta(x_2,x_{12}x_1)=d_{x_2x_{12}}(x_2x_{12},x_2x_{12}x_1)$.
    \item $u_x(x_{2}x_{12},x_2x_{12}x_1)=d_x(x_{2}x_{12},x_2x_{12}x_1)=0$ for $x\neq x_2x_{12}$.
\end{itemize}
  \end{enumerate}
\end{description}

\begin{description}
  \item[$gr(y)+gr(z)=8$]

   \

\begin{itemize}
	\item[(i)]   We obtain zero in both sides of \eqref{eqn:conm1}, hence $u_x(y,z)=d_x(y,z)=0$.
	\end{itemize}
\end{description}

The other cases are very similar. Now, we have that $\eta$ verifies \eqref{eqn:conm1} and the hypothesis of the Lemma  \ref{lem:conmutation}; hence $\eta$ verifies \eqref{eqn:conm2} and $e^{\eta}$ is multiplicative.
\epf

\begin{lemma}\label{lem:alpha-sigma=eta}
Let $\sigma=\sigma_{\bs\lambda}$ be as in Theorem \ref{thm:sigma-a2} and $\eta\in \Z^2(\Bq_{\bq},\k)^H$ as in \eqref{eq:eta-abeliano}.

Assume that $e^{\eta}$ is a Hopf cocycle, cohomologous to $\sigma$. Then
\begin{align}\label{eq:eta-de-alpha}
  \lambda_1=\eta_1,\quad\lambda_2=\eta_2,\quad\eta_{121}+\eta_{212}=\lambda_{12}+2q_{12}\lambda_1\lambda_2.
\end{align}
Moreover, if $\alpha\rightharpoonup\sigma=e^{\eta}$,  $\alpha\in U(\Bq_{\bq}^{\ast})$, then $\alpha(1)=1$ and
\begin{align}\label{eq:alpha-necessary}
\begin{split}
\alpha(x_1)&=\alpha(x_2)=\alpha(x_{12})=\alpha(x_2x_1)=\alpha(x_{12}x_1)=\alpha(x_2x_{12})=0,\\
\alpha(x_2x_{12}x_1)&=\eta_{212}-\lambda_{12}.
\end{split}
\end{align}
\end{lemma}
\pf
Suppose that there exist $\alpha\in U(\Bq_{\bq}^\ast)$ such that $\alpha\rightharpoonup \sigma=e^\eta$. In first place, we may define $\alpha(1)=1$ and then Remark \ref{rem:alpha-grado1} gives us $\lambda_1=\eta_1$ and $\lambda_2=\eta_2$. In particular, it follows that
\begin{itemize}
    \item $\alpha(x_1)=\alpha(x_2)=0$.
    \item $\alpha(x_ix_j)=0$, for all $i,j\in \I_2$ with $i\neq j$.
\end{itemize}
On the other hand, from the relations
\begin{align*}
\alpha\rightharpoonup \sigma(x_i,x_jx_i)&=e^{\eta}(x_i,x_jx_i)=0, & \alpha\rightharpoonup \sigma(x_i,x_{12})&=e^{\eta}(x_i,x_{12})=0,
\end{align*}
we obtain that $\alpha^{-1}(x_ix_jx_i)=0$, so $\alpha(x_ix_jx_i)=0$. Now we compute $\alpha(x_2x_{12}x_1)$, expanding $\alpha\rightharpoonup \sigma(x_1,x_2x_{12})=e^\eta(x_1,x_2x_{12})$ and $\alpha\rightharpoonup \sigma(x_2,x_{12}x_1)=e^\eta(x_2,x_{12}x_1)$. For the first,
\begin{align*}
\alpha\rightharpoonup \sigma(x_2,x_{12}x_1)&=\eta_{121}\\
\sigma(x_2,x_{12}x_1)+\alpha^{-1}(x_2x_{12}x_1)&=\eta_{121}\\
\alpha^{-1}(x_2x_{12}x_1)&=\eta_{121}-2q_{12}\lambda_1\lambda_2.
\end{align*}
For the second, we obtain:
\begin{align*}
\alpha\rightharpoonup \sigma(x_1,x_{2}x_{12})&=\eta_{212}\\
\sigma(x_1,x_{2}x_{12})-\alpha^{-1}(x_2x_{12}x_1)&=\eta_{212}\\
\alpha^{-1}(x_2x_{12}x_1)&=\lambda_{12}-\eta_{212}.
\end{align*}
Hence, as $\alpha^{-1}(x_2x_{12}x_1)$ exists, then $\eta_{121}+\eta_{212}=\lambda_{12}+2q_{12}\lambda_1\lambda_2$.
%
\epf

Let $\overline{\C}\subset\Z^2(\Bq_{\bq},\k)^\Gamma$ be the set defined by
\begin{align*}
\overline{\C}\coloneqq\{\eta_i\xi_i^i+\eta_{121}\xi_{121}^2-\eta_{121}\xi_{212}^1,\eta_{121}\xi_{121}^2+\eta_{212}\xi_{212}^1 :\eta_i,\eta_{121},\eta_{212}\in\k,i\in\I_2\}.
\end{align*}
Notice that $\eta\in \overline{\C}$ if and only if
\begin{align}\label{eq:eta-de-hopf}
\eta_1\eta_2=\eta_1(\eta_{121}+\eta_{212})=\eta_2(\eta_{121}+\eta_{212})=0.
\end{align}

\begin{remark}\label{rem:commutation-not-necessary}
In particular, $\C\subset \overline{\C}$ and this inclusion is strict. However, they coincide up to cobordisms, see \eqref{eqn:cobordism}:
\[
\overline{\C}=\C+\k\{\xi_{121}^2-\xi_{212}^1\}.
\]
\end{remark}

\medbreak

We have the following characterization of Hopf cocycles that are an exponential map of a Hochschild cocycle.

\begin{lemma}\label{lem:converse}
	$e^\eta$ is a Hopf 2-cocycle if and only if $\eta \in \overline{\C}$. 
\end{lemma}
\pf
Let $\eta\in \Z^2(\Bq_{\bq},\k)^\Gamma$ be a Hochschild cocycle as in \eqref{eq:eta-abeliano}.
Assume $e^\eta$ is a Hopf 2-cocycle. We compute the values of $e^\eta$ in two ways. Namely, via the decomposition formulas in Example \ref{exa:descompos-lowdregree} and via the definition of an exponential map. We obtain a list of equations that imply that $\eta \in \overline{\C}$. We show below some of these calculations.
\begin{description}[leftmargin=*]
\item[$e^{\eta}(x_2x_1,x_{12})$] Using the decomposition formula, we obtain
\[
e^{\eta}(x_2x_1,x_{12})=2\eta_1\eta_2-q_{12}\eta_{121},
\] but the definition of exponential map gives
\[
e^{\eta}(x_2x_1,x_{12})=\eta_1\eta_2-q_{12}\eta_{121}.
\]
Thus, $\eta_1\eta_2=0$. From now on, we suppose that this relations hold.
\item[$e^{\eta}(x_{12}x_1,x_{12}x_1)$] We obtain:
\[
e^{\eta}(x_{12}x_1,x_{12}x_1)=q_{12}\eta_1\eta_{121}+q_{12}\eta_1\eta_{212},  \qquad
e^{\eta}(x_{12}x_1,x_{12}x_1)=0.
\]
So, we get the relation $\eta_1(\eta_{121}+\eta_{212})=0$.
\item[$e^{\eta}(x_2x_{12},x_2x_{12})$] Here we get
\[
\sigma(x_2x_{12},x_2x_{12})=-q_{12}\eta_2\eta_{121}-q_{12}\eta_2\eta_{212},
\qquad
e^{\eta}(x_2x_{12},x_2x_{12})=0.
\]
Hence $\eta_2(\eta_{121}+\eta_{212})=0$.
\end{description}
Therefore, we see that $\eta\in\overline{\C}$.

Conversely, suppose that $\eta\in\overline{\C}$. If $\eta=\eta_{121}\xi_{121}^2+\eta_{212}\xi_{212}^1$, then $e^{\eta}$ is a Hopf cocycle by Proposition \ref{pro:exponencial-abeliano}. Now, suppose that $\eta=\eta_1\xi_1^1+\eta_{121}\xi_{121}^2-\eta_{121}\xi_{212}^1$. The case $\eta=\eta_2\xi_2^2+\eta_{121}\xi_{121}^2-\eta_{121}\xi_{212}^1$ is analogous.

We claim that, by definition, $e^{\eta}$ is given by the table
\begin{center}
	\begin{table}[h]
		\resizebox{10cm}{!}
		 {\begin{tabular}{|c|c|c|c|c|c|c|c|c|}
				\hline
				$e^\eta$ &  $x_1$ & $x_2$ & $x_{12}$ & $x_2x_1$  &  $x_2x_{12}$ & $x_{12}x_1$ & $x_2x_{12}x_1$ \\
				\hline
				$x_1$   & $\eta_1$ & $0$ & $0$ & $0$ & $-\eta_{121}$ & $0$ & $0$ \\
				\hline
				$x_2$  & $0$ & $0$ & $0$ & $0$  &  $0$ & $\eta_{121}$& $0$\\
				\hline
				$x_{12}$  & $0$ & $0$ & $0$ & $-q_{12}\eta_{121}$ &  $0$ & $0$ & $0$ \\
				\hline
				$x_2x_1$  & $0$ & $0$ & $-q_{12}\eta_{121}$ & $\eta_{121}$  &  $0$ & $0$ & $0$   \\
				\hline
				$x_2x_{12}$  & $\eta_{121}$ & $0$ & $0$ & $0$& $0$ & $0$ & $0$  \\
				\hline
				$x_{12}x_1$  & $0$ & $-\eta_{121}$ & $0$ & $0$ &  $0$ & $0$ & $0$  \\
				\hline
				$x_2x_{12}x_1$ & $0$ & $0$ & $0$ & $0$  &  $0$ & $0$ & $-\eta_{121}^2$ \\
				\hline
		\end{tabular}}
	\end{table}
\end{center}
Indeed, by Remark \ref{rem:minimo} and a degree argument, plus considering that $\eta_1\eta_2=0$, we see that $e^{\eta}(x,y)=\eta(x,y)$ for $\deg x+\deg y\leq4$.

On the other hand, if $\deg x+\deg y=5$, then
\begin{align*}
  e^\eta(x,y)=\eta(x,y)+\eta(x_{\underline{{(1)}}},y_{\underline{{(1)}}})\eta(x_{\underline{{(2)}}},y_{\underline{{(2)}}}),
\end{align*}
so $e^\eta(x,y)=0$ in this case.
Now, when $\deg x+\deg y=6$ and $\deg x\neq\deg y$, it follows that for each of the four possible cases, there are scalars $(a_{ij})_{i,j}$, depending on $x,y$, such that the following relation holds:
\begin{align*}
  e^\eta(x,y)=\eta(x,y)+\sum_{i,j\in\I_2,i\leq j}a_{ij}\eta(x_i,x_i)\eta(x_j,x_{ij}x_j).
\end{align*}
Thus $e^\eta(x,y)=0$ as $\eta(x,y)=0$ and $\eta(x_j,x_{ij}x_j)=q_{ij}\eta(x_j,x_jx_{ij})=0$.
If $\deg x=\deg y=3$, then it follows that $e^\eta(x,y)=0$.

If $\deg x+\deg y=7$, then we see that $e^\eta(x,y)=0$. Finally, if $x=y=x_2x_{12}x_1$, then it is easy to see that $\eta^{\ast k}(x,y)=0$ for $k\neq 2$ and
\begin{align*}
\eta^{\ast 2}(x,y)=&\, -4\eta(x_2x_1,x_2x_1)^2 +4q_{21}\eta(x_2x_1,x_{12})\eta(x_2x_1,x_2x_1)\\
&+\eta(x_1,x_2x_{12})\eta(x_2x_{12},x_1)+2q_{21}\eta(x_{12},x_2x_1)\eta(x_2x_1,x_2x_1)\\
&-\eta(x_{12},x_{12})\eta(x_2x_1,x_2x_1)-\eta(x_{12},x_2x_1)\eta(x_2x_1,x_{12})\\
&+2\eta(x_{12}x_1,x_2)\eta(x_2,x_{12}x_1)+2q_{21}\eta(x_2x_1,x_2x_1)\eta(x_{12},x_2x_1)\\
&-\eta(x_2x_1,x_{12})\eta(x_{12},x_2x_1)-\eta(x_2x_1,x_2x_1)\eta(x_{12},x_{12})\\
&+\eta(x_2x_{12},x_1)\eta(x_1,x_2x_{12})\\
=&\, -4\eta_{121}^2+4\eta_{121}^2+\eta_{212}\eta_{121}+2\eta_{121}^2-\eta_{121}^2-\eta_{212}\eta_{121}-\eta_{121}^2\\
&+2\eta_{212}\eta_{121}+2\eta_{121}^2-\eta_{121}^2-\eta_{121}^2-\eta_{212}\eta_{121}+\eta_{121}\eta_{212}\\
=&\, 2\eta_{212}\eta_{121}.
\end{align*}
Hence, recalling that $\eta_{212}=-\eta_{121}$, we obtain that $e^\eta(x,y)=\frac{1}{2}\eta^{\ast 2}(x,y)=-\eta_{121}^2$.

Now, let $\sigma=\sigma_{\eta_1,0,0}$ as in Theorem \ref{thm:sigma-a2} and $\alpha\in U(\Bq_{\bq}^{\ast})$ as in \eqref{eq:alpha-necessary}.
Then $\alpha\rightharpoonup\sigma_{\eta_1,0,0}(x_i,b)=e^\eta(x_i,b)$, for any $i\in\I_2$, $b\in \mathbb{B}_\bq$, see Lemma \ref{lem:alpha-sigma=eta}. Next, we apply the decomposition formulas to compute all values of the Hopf cocycle $\alpha\rightharpoonup\sigma_{\eta_1,0,0}$ and we obtain that $\alpha\rightharpoonup\sigma=e^{\eta}$. Therefore, $e^{\eta}$ is a Hopf cocycle.
\epf

Let $\ex(\C), \ex(\overline{\C})\subset Z^2(\Bq_{\bq})$ be the subsets of Hopf cocycles that arise as exponential maps of elements of $\C$ and $\overline{\C}$ respectively. Next, we show that it is enough to consider exponential cocycles in $\ex(\C)$.

\begin{lemma}\label{lem:c-eta-enough}
Let $\sigma_{\bs\lambda}\in Z^2(\Bq_{\bq})$ be as in Theorem \ref{thm:sigma-a2}. Then $\sigma$ is cohomologous to a cocycle in $\ex(\C)$ if and only $\sigma$ is cohomologous to a cocycle in $\ex(\overline{\C})$.
\end{lemma}
\pf
One implication is trivial. Assume $\sigma\sim_\alpha e^\eta$, $\eta\in \overline{\C}$. That is, $\alpha\ra\sigma=e^\eta$.

If $\eta=\eta_{121}\xi_{121}^2+\eta_{212}\xi_{212}^1$, then $\eta\in \C$  and we can take $\alpha=\eps$.

Fix now $\eta=\eta_1\xi_1^1+\eta_{121}\xi_{121}^2-\eta_{121}\xi_{212}^1$.
By Lemma \ref{lem:alpha-sigma=eta}, (notice that $\eta_{212}=-\eta_{121}$) if we define $\alpha'$ by
\begin{align*}
\alpha'(1)=1,\quad \alpha'(x)=0, \text{ for $0<\deg x\leq 3$},\quad\alpha'(x_2x_{12}x_1)=\eta_{121},
\end{align*}
then $\alpha'\ra e^\eta=\sigma_{\eta_1,0,0}$. But $\sigma_{\eta_1,0,0}=e^{\eta'}$ and $\eta'=\eta_1\xi_1^1\in \C$. Hence $\sigma\sim_\beta e^{\eta'}$, for $\beta=\alpha'\ast\alpha$.
\epf


The next theorem characterizes the cocycles that are co-homologous to an exponential map of a Hochschild cocycle.

\begin{theorem}\label{thm:sigma-pure}
Let $\sigma=\sigma_{\bs\lambda}$ be a Hopf 2-cocycle as in Theorem \ref{thm:sigma-a2}, $\sigma\neq\eps$. Then the following are equivalent:
\begin{enumerate}[label=(\alph*)]
    \item $\sigma\sim e^\eta$ for some $\eta\in \Z^2(\Bq_{\bq},\k)^H$.
    \item There is a unique no null $\lambda\in\{\lambda_1,\lambda_2,\lambda_{12}\}$.
    \item $\sigma\in \ex(\Z^2(\Bq_{\bq},\k)^H)$
\end{enumerate}
In particular, any $\sigma=\sigma_{\lambda_1,\lambda_2,\lambda_{12}}$ with at least two nonzero parameters is pure.
\end{theorem}
\pf
First, notice that we can assume $\eta\in \C$, by Lemma \ref{lem:c-eta-enough}.

$(a)\Rightarrow (b)$.
Combining Lemma \ref{lem:alpha-sigma=eta} with equations \eqref{eq:eta-Ceta}, we obtain:
\[
\lambda_1\lambda_2=0,\qquad\lambda_1\lambda_{12}=\lambda_2\lambda_{12}=0.
\]

$(b)\Rightarrow (c)$. By looking at the tables in Theorem \ref{thm:sigma-a2} and in the proof of Lemma \ref{lem:converse}, we see that   this implication is clear as $\sigma_{(\lambda,0,0)}=e^{\lambda\xi_1^1}$. Similarly, $\sigma_{(0,\lambda,0)}=e^{\lambda\xi_2^2}$. Finally, if $\eta=\lambda\xi_{212}^1$, it follows from Lemma \ref{lem:alpha-sigma=eta} that $\sigma_{(0,0,\lambda)}=e^\eta$, as $\alpha=\eps$ therein.

$(c)\Rightarrow (a)$. Clear.
\epf

\begin{remark}
In the setting of Lemma \ref{lem:alpha-sigma=eta}, we see that if $e^\eta\sim \eps$, then $\eta_1=\eta_2=0$ and $\eta_{121}=-\eta_{212}$. In other words, $\eta\in \B^2(B,\k)$ as it is a scalar multiple of $\beta\coloneqq\xi_{121}^2-\xi_{212}^1$, see \eqref{eqn:cobordism}. Conversely, $e^\beta$ is a Hopf cocycle by Proposition \ref{pro:exponencial-abeliano} and Lemma \ref{lem:alpha-sigma=eta} shows that $e^\beta\sim\sigma_{(0,0,0)}=\eps$.
	
We remark that $e^\beta\neq \eps$, as shown in the table in the proof of Lemma \ref{lem:converse}.
\end{remark}

%
%
%
%

\end{document}